\documentclass[12pt,a4paper]{article}
\usepackage{theorem}
\usepackage{graphicx}
\newif\ifpics
\picstrue %% display pictures (comment line to switch off pictures)

\newtheorem{lemma}{Lemma}
\newtheorem{proposition}[lemma]{Proposition}
\newtheorem{theorem}[lemma]{Theorem}
\newtheorem{corollary}[lemma]{Corollary}
\newtheorem{exampl}[lemma]{Example}

\newcommand{\C}{{\bf C}}

\newcommand{\N}{{\bf N}}

\newcommand{\R}{{\bf R}}

\newcommand{\Z}{{\bf Z}}
\newcommand{\rme}{{\rm e}}

\newcommand{\ri}{{\rm i}}
\newcommand{\cE}{{\cal E}}

\newcommand{\cH}{{\cal H}}

\newcommand{\sig}{\sigma}
\newcommand{\alp}{\alpha}
\newcommand{\bet}{\beta}
\newcommand{\gam}{\gamma}

\newcommand{\lam}{\lambda}
\newcommand{\del}{\delta}
\newcommand{\eps}{\varepsilon}

\newcommand{\Gam}{\Gamma}
\newcommand{\Ome}{\Omega}
\newcommand{\tr}{{\rm tr}}

\newcommand{\Spec}{{\rm Spec}}

\newcommand{\norm}{\Vert}
\renewcommand{\Re}{{\rm Re}}
\renewcommand{\Im}{{\rm Im}}

\newcommand{\Num}{{\rm Num}}

\newcommand{\implies}{\Rightarrow}
\newcommand{\wdh}{\widehat}
\newcommand{\eqref}[1]{(\ref{#1})}
\newenvironment{proof}{\textbf{Proof}}{\hfill$\Box$}

\newenvironment{choices}{ \left\{ \begin{array}{ll} }{\end{array}\right.}
\newcommand{\mtrx}[1]{\left(\begin{array}{cc}#1 \end{array}\right)}
\newcommand{\txtmtrx}[1]{\raisebox{0.25ex}{\scalebox{0.6}{$\mtrx{#1}$}}}
\newcommand{\vctr}[1]{\left(\begin{array}{c}#1 \end{array}\right)}

\newcommand{\bkt}[1]{ [ \hspace*{-0.15em} [ #1 ] \hspace*{-0.15em} ] }
\newcommand{\move}[1]{}
\usepackage{amssymb}
\usepackage{hyperref}
\usepackage{graphicx}
\title{Spectrum of a Feinberg-Zee\\
Random Hopping Matrix}
\author{S N Chandler-Wilde\\
E B Davies}
\date{December 2011}
\begin{document}
\maketitle
%%%%%%%%%%%%%%%%%%%%%
\begin{abstract}
This paper provides a new proof of a theorem of Chandler-Wilde,
Chonchaiya and Lindner that the spectra of a certain class of
infinite, random, tridiagonal matrices contain the unit disc almost
surely. It also obtains an analogous result for a more general class
of random matrices whose spectra contain a hole around the origin.
The presence of the hole forces substantial changes to the analysis.
\end{abstract}



\section{Introduction}\label{intro}

Over the last fifteen years there have been many studies of the
spectral properties of non-self-adjoint, random, tridiagonal
matrices $A$, some of them cited in \cite{CCM,FZ1,FZ2,HOZ}. It has
become clear that if all of the off-diagonal entries $A_{i,j}$ with $i-j=\pm 1$ of the matrices concerned are positive, the almost sure limit as
$N\to\infty$ of the spectra of random $N\times N$ matrices subject
to periodic boundary conditions can be quite different from the spectral
behaviour of the corresponding infinite random matrix,
\cite{EBD0,EBD1,GK1,GK2}. Indeed the limit in the first case can be the union of a small number of simple curves, while the second limit has a non-empty interior.

Numerical calculations suggest that the situation is quite different if the off-diagonal entries have variable signs, but much less has been proved in this situation, which is the one that we consider here. In a recent paper, \cite{CCL}, Chandler-Wilde, Chonchaiya and Lindner made important progress in determining the almost sure spectrum of a remarkably interesting class of non-self-adjoint, random, tridiagonal
matrices introduced by Feinberg and Zee in \cite{FZ1}, and sometimes
called random hopping matrices, because the diagonal entries all
vanish. Specifically they proved that, contrary to earlier
conjectures, the infinite, tridiagonal matrix
\[
A_c= \left(\begin{array}{cccccc}
\ddots&\ddots&&&&\\
\ddots&0&1&&&\\
&c_{n-1}&0&1&&\\
&&c_n&0&1&\\
&&&c_{n+1}&0&\ddots\\
&&&&\ddots&\ddots\\
\end{array}\right)
\]
has spectrum that contains the unit disc almost surely, \cite{CCL}.
The paper assumed that the entries $c_n$ are independent and
identically distributed with values in $\{\pm 1\}$.

In the present paper we assume that the entries $c_n$ are
independent and identically distributed with values in $\{\pm
\sig\}$ for some fixed $\sig\in (0,1]$. We assume that the
probability $p$ that $c_n=\sig$ satisfies $0<p<1$; the corresponding
probability measure on $\Ome_\sig=\{\pm \sig\}^\Z$ is denoted by
$\mu$. The matrix $A_c$ is identified with the bounded operator
acting in the natural manner on $\ell^2(\Z)$.

In Lemma~\ref{lemmaA} we prove that
\[
\Spec(A_c)\subseteq \{ \lam: 1-\sig\leq |\lam|\leq 1+\sig\}
\]
by a perturbation argument. We also prove that
\[
\Spec(A_c)\subseteq \{ x+iy:|x|+|y|\leq \sqrt{2(1+\sig^2)}\}
\]
by obtaining a bound on the numerical range of $A_c$. There are currently no general techniques for identifying the precise forms of holes in the spectra of non-self-adjoint operators, and we have not done so here, but numerical calculations are consistent with the hypothesis that it is the intersection, $H_\sig$, of two elliptical
regions as defined in (\ref{Hsigma}); see the figures at the end of Section \ref{firstproof}. Little is known about the part
of the spectrum of $A_c$ outside the unit disc even in the case
$\sig=1$, but numerical studies suggest that the boundary of the
spectrum has a self-similar fractal structure in that case;
\cite{CCL,HOZ}.

The main result of \cite{CCL}, that the spectrum contains the unit disc almost surely, is for the case that $\sig=1$, when there
is no hole in the spectrum. It depends upon the identification of a
particular sequence $c\in \Ome_1$ such that the equation $A_cf=\lam
f$ has a bounded solution $f$ for every $\lam\in \C$ such that
$|\lam |<1$.

Our Theorem~\ref{maintheorem} rederives the main result of
\cite{CCL}, in which $\sig=1$, but depends on a certain operator
identity introduced in the next section. Our main result,
Theorem~\ref{maintheorem3}, that the spectrum of $A_c$ contains that part of the unit disc which is not in $H_\sig$,
applies to all $\sig\in (0,1)$. We give
a second proof of this result in Theorem~\ref{maintheorem6}, by
combining some results for $\sig=1$ with bounds on the Lyapunov
exponents of certain transfer matrices. Both proofs depend, additionally,
on results on the spectra of operators on $\ell^2(\Z)$ which have different periodic structures on the positive and negative half-axes. They also both depend on explicit spectral calculations which we are able to carry out for certain operators $A_c$ with $c$ having arbitrarily large period.

Our main results, as just stated, concern the spectrum of the (bi-)infinite matrix $A_c$. In a shorter final section we spell out implications for the spectra of the corresponding semi-infinite and finite matrices, illustrating these observations with computations of the  finite matrix spectra. In particular we show, by applying recent results of Lindner and Roch \cite{LR10}, that, unlike $A_c$, the semi-infinite matrix has no hole in its spectrum for $\sig\in (0,1)$, but contains the unit disc for all $\sig\in (0,1]$.

Let $\cE_\sig$ denote the set of all $c\in\Ome_\sig$ that are
pseudo-ergodic in the sense of \cite{EBD1}. Precisely, $c\in\cE_\sig$
if for every finite sequence $b:\{1,...,n\}\to\{\pm \sig\}$ there
exists $m\in\Z$ such that $b_r=c_{m+r}$ for all $r\in \{1,...,n\}$.
Such sequences $c$ are easy to construct without any reference to
probability theory. The following facts, proved in \cite{EBD1}, and rederived in \cite{Li,CWL} as an instance of the application of limit operator arguments, will
be crucial in this paper.

\begin{proposition}\label{quasiergodic}
If $b,\, c \in\cE_\sig$ then $\Spec(A_b)=\Spec(A_c)$. Let $S_\sig$
denote this set, which is the main object of study in the paper. If
$c\in\Ome_\sig$, then $c\in\cE_\sig$ almost surely with respect to
the measure $\mu$. Finally
$$
S_\sig = \bigcup_{b\in\Ome_\sig} \Spec ( A_b ).
$$
\end{proposition}

To describe a further result we establish, for $N\in \N$ and $\sigma\in (0,1]$ let $\pi_{N,\sigma}$ denote the union of $\Spec(A_c)$ over all $c\in \Omega_\sigma$ that are periodic with period $\leq N$. Let
\begin{equation} \label{eq:piinf}
\pi_{\infty,\sig} =  \bigcup_{N\in\N} \pi_{N,\sig}.
\end{equation}
One obvious implication of the above proposition is that
\begin{equation} \label{eq:pidef}
\pi_{\infty,\sig} \subset S_\sigma.
\end{equation}
As is well-known, the set $\pi_{N,\sigma}$ is the union of eigenvalues of $N\times N$ matrices. (Precisely, it is the union, over all sequences $c$ and all $|\alpha|=1$, of the eigenvalues of the matrix $A^{(N,\mathrm{per})}_{c,\alpha}$ defined in \eqref{finitematrix2} below; see \eqref{spectrumperiodic} and \cite{LOTS}. For another, equivalent characterisation see Lemma \ref{BIO}.) This simple observation is useful, in that it provides a method for computing what prove to be large subsets of $S_\sigma$, and will be one component in our arguments.

An interesting question is whether $\pi_{\infty,\sig}$ is dense in $S_\sig$. We do not answer this question one way or the other, but our method of proof of Theorem~\ref{maintheorem}, showing that the unit disc is a subset of $S_1$, as a by-product, and with some additional argument, leads to a proof that $\pi_{\infty,1}$ is dense in the unit disc (Theorem \ref{denseness}).

For the sake of simplicity we will, throughout the rest of the paper, omit the subscript $\sig$ in our notations if $\sig=1$. We use $\N$  and $\Z_+$, respectively, as our notations for the sets of positive and non-negative integers.

\section{An abstract theorem}\label{abstract}

In this section we present an abstract theorem that might be
interesting in other contexts. It will be applied in
Section~\ref{main}.

Let $A$ be a bounded linear operator acting on the Hilbert space
$\cH$ and let $\cH=\cH_e\oplus\cH_o$ be an orthogonal decomposition of $\cH$.

\begin{lemma}\label{interchange}
If $A(\cH_e)\subseteq \cH_o$ and $A(\cH_o)\subseteq \cH_e$ then $\cH_e$ and $\cH_o$ are invariant
under the action of $A^2$. If $B$ is the restriction of $A^2$ to
$\cH_e$ and $M$ is the restriction of $A^2$ to $\cH_o$ then
\begin{equation}
\Spec(A^2)\backslash \{0\}=\Spec(B)\backslash
\{0\}=\Spec(M)\backslash \{0\}.\label{identity1}
\end{equation}
If $A$ is invertible then
\begin{equation}
\Spec(A^2)=\Spec(B)=\Spec(M).\label{identity2}
\end{equation}
\end{lemma}

\begin{proof}
The decomposition $\cH=\cH_e\oplus\cH_o$ allows one to write the operator $A$ in the form
\[
A=\mtrx{0&X\\Y&0}
\]
where $X:\cH_o\to\cH_e$ and $Y:\cH_e\to\cH_o$. Therefore
\begin{equation}
A^2=\mtrx{XY&0\\0&YX}.\label{A2identity}
\end{equation}
This implies that $B=XY$ and $M=YX$. The second identity in (\ref{identity1}) follows by some simple algebra that holds for any pair of bounded operators $X$ and $Y$, and the first identity is a trivial consequence.

If $A$ is invertible then (\ref{A2identity}) implies that $B$ and $M$
are also invertible; therefore (\ref{identity2}) is equivalent to
(\ref{identity1}).
\end{proof}

\begin{theorem}\label{spectralequality}
Let $\cH=\ell^2(\Z)$, let $\cH_e$ be the closed subspace of sequences whose supports are contained in the set of even integers, and let $\cH_o$ be the closed subspace of sequences whose supports are
contained in the set of odd integers. Let $A$ be a bounded operator
on $\cH$ whose matrix satisfies $A_{r,s}= 0$ for all $r,\, s$ such
that $|r-s|\not= 1$. Then $A(\cH_e)\subseteq \cH_o$ and $A(\cH_o)\subseteq \cH_e$. Moreover the identities
\[
\Spec(A^2)=\Spec(B)=\Spec(M)
\]
are valid in either of the following two cases.
\begin{enumerate}
\item
$|A_{r,s}|=1$ for all $r,\, s$ such that $|r-s|=1$;
\item
There exist constants $\bet,\, \gam$ such that $0<\bet<\gam<\infty$
and $|A_{r,s}|\leq \bet$ if $r-s=1$ while $|A_{r,s}|\geq \gam$ if
$r-s=-1$.
\end{enumerate}
\end{theorem}

\begin{proof}\\
\textbf{Case~1.} An elementary calculation establishes that there
exists a sequence $f:\Z\to\C$ such that $Af=0$, $|f_{2n}|=1$ for all
$n$ and $f_{2n+1}=0$ for all $n$, so that $A$ and $B$ are not invertible viewed as operators on $\ell^\infty(\Z)$, and thus not invertible as operators on $\ell^2(\Z)$ (see e.g.\ \cite[Theorem 2.5.2]{RRS}). So $0\in\Spec(A)$ and
$0\in\Spec(B)$. Similarly there exists a sequence $f:\Z\to\C$ such
that $Af=0$, $|f_{2n+1}|=1$ for all $n$ and $f_{2n}=0$ for all $n$.
Hence $0\in\Spec(M)$. The result follows by combining this with
(\ref{identity1}).

\textbf{Case~2.} The operator $A_L$ associated with the matrix
\[
(A_L)_{r,s}=\begin{choices}
A_{r,s}&\mbox{ if $r-s=-1$},\\
0&\mbox{ otherwise},
\end{choices}
\]
is invertible and satisfies $\norm A_L^{-1}\norm\leq \gam^{-1}$. The
operator $A_R=A-A_L$ satisfies $\norm A_R\norm\leq \bet$. Therefore
$A$ is invertible with
\[
\norm A^{-1}\norm =\norm A_L^{-1}(I+A_RA_L^{-1})^{-1}\norm %
\leq \frac{\gam^{-1}}{1-\bet/\gam}=\frac{1}{\gam-\bet}.
\]
The proof is completed by applying \eqref{identity2}.
\end{proof}

\section{The case $\sig=1$}\label{main}

The following lemma was noted in \cite{CCL}.

\begin{lemma}\label{symmetry}

If $c\in\Ome$ then $\Spec(A_c)$ is invariant with respect to both of
the maps $\lam\to\overline{\lam}$ and $\lam\to -\lam$. If $\lam\in
S$ then $\overline{\lam}$ and $i\lam$ lie in $S$. Hence $S$ is
invariant under the dihedral symmetry group $D_2$ generated by these
two maps.
\end{lemma}

\begin{proof}
The invariance of $\Spec(A_c)$ under complex conjugation follows
directly from the fact that $A_c$ has real entries. If $D$ is the
diagonal matrix with entries $D_{r,r}=(-i)^r$ for all $r\in\Z$ then
$DA_cD^{-1}=iA_{-c}$, so
\begin{equation}
\Spec(A_c)=i\Spec(A_{-c}).\label{irotate}
\end{equation}
Iterating this identity yields $\Spec(A_c)=-\Spec(A_{c})$. This
proves the first part of the lemma. The second part follows once one
observes that $c\in  \cE$ if and only if $-c\in\cE$.
\end{proof}

The formulae in (\ref{iteration}) are related to those in
\cite[Prop.~2.1]{CCL}, in a way that we will make explicit in Section \ref{sec:maps}. However, nothing resembling the following lemma appears in \cite{CCL}.

\begin{lemma}\label{square} Given $b\in\Ome$, let $c=\Gam_+(b)\in\Ome$ be the unique sequence satisfying
\begin{equation}
c_0=1, \hspace{2em} c_{2n}+c_{2n+1}=0, \hspace{2em}
c_{2n}c_{2n-1}=b_n,\label{iteration}
\end{equation}
for all $n\in \Z$. Then $A_c^2$ is unitarily equivalent to
$A_b\oplus M_b$ acting in $\ell^2(\Z)\oplus \ell^2(\Z)$, where
\begin{equation}
(M_bf)_n=-f_{n-1}+(c_{2n+1}+c_{2n+2})f_n+f_{n+1}\label{Mbformula}
\end{equation}
for all $f\in \ell^2(\Z)$. Moreover
\[
\Spec(A_c^2) = \Spec(A_b)=\Spec(M_b).
\]
\end{lemma}

\begin{proof}
One may write $(A_cf)_n=c_nf_{n-1}+f_{n+1}$ for all $n\in \Z$, or
equivalently $A_c=V_cR+L$ where $(Lf)_n=f_{n+1}$, $(Rf)_n=f_{n-1}$
and $(V_cf)_n=c_nf_n$ for all $f\in \ell^2(\Z)$.

Therefore
\begin{eqnarray*}
A_c^2&=& V_cRV_cR+LV_cR+V_cRL+L^2\\
&=& X_cR^2+Y_c+L^2
\end{eqnarray*}
where $X_c$ and $Y_c$ are the diagonal matrices with diagonal
entries
\begin{eqnarray*}
X_{c,n,n}&=&c_nc_{n-1}, \\
Y_{c,n,n}&=& c_n+c_{n+1}.
\end{eqnarray*}
The operator $A_c^2$ has two invariant subspaces
\[
\cH_e=\{f\in \ell^2(\Z):f_{2n+1}=0 \mbox{ for all } n\in\Z\}
\]
and $\cH_o=\ell^2(\Z)\ominus \cH_e$. After an obvious relabeling of
the subscripts, the restriction of $A_c^2$ to $\cH_e$ equals $A_b$
while the restriction of $A_c^2$ to $\cH_o$ is equal to $M_b$, as
defined in (\ref{Mbformula}). The final statement of the lemma is
now an application of Theorem~\ref{spectralequality}, case~1.
\end{proof}

We will exploit extensively the formula $\Spec(A_c^2) = \Spec(A_b)$ which appears in the above lemma. The equation $\Spec(A_b)=\Spec(M_b)$ will not play a role in our subsequent arguments, but makes an intriguing connection between spectra of rather different tridiagonal operators. Extending this connection slightly,
for $b\in \Ome$ define $c=\Gam_+(b)$ and $\tilde M_b$ by
$$
(\tilde M_b f)_n =  f_{n-1} + i^n(c_{2n+1}+c_{2n+2}) f_n + f_{n+1},
$$
for all $f\in \ell^2(\Z)$. Then, arguing as we do above to show \eqref{irotate}, we see that
$$
\Spec(\tilde M_b) = i\,\Spec(M_b).
$$
In particular, in the case $b\in \cE$ when, by Lemma \ref{symmetry}, $i\Spec(M_b)=i\Spec(A_b) = \Spec(A_b)$, we see that
$$
S = \Spec(A_b) = \Spec(\tilde M_b).
$$
Thus, in studying $S$, we are studying both the almost sure spectrum  of the infinite hopping-sign matrix $A_b$ with respect to the measure $\mu$, and the almost sure spectrum, with respect to the same measure, of $\tilde M_b$, a discrete Schr\"odinger operator with a particular, complex random potential.

In the next lemma we define the square root of any non-zero complex
number to be the root whose argument lies in $(-\pi/2,\pi/2]$.

\begin{lemma}\label{squareroot}
If $b\in\Ome$ and $c=\Gam_+(b)$ then $\lam\in\Spec(A_b)$ if and only
if $\pm\sqrt{\lam}$ both lie in $\Spec(A_c)$. If $\lam\in S$ then
$\pm \sqrt{\lam}$ both lie in $S$.
\end{lemma}

\begin{proof}
Lemma~\ref{symmetry} and Lemma~\ref{square} imply that the following
statements are equivalent. $\lam\in \Spec(A_b)$; $\lam\in
\Spec(A_c^2)=\left( \Spec(A_c)\right)^2$; either $\sqrt{\lam}$ or
$-\sqrt{\lam}$ lies in $\Spec(A_c)$; $\pm\sqrt{\lam}$ both lie in
$\Spec(A_c)$.

If $\lam\in S$ and $b\in \cE$ then $\lam\in \Spec(A_b)$ by
Proposition~\ref{quasiergodic}. Lemma~\ref{square} implies that
\[
\lam\in \Spec(A_c^2)=\left(\Spec(A_c)\right)^2\subseteq S^2.
\]
Therefore either $\sqrt{\lam}$ or $-\sqrt{\lam}$ lie in $S$. The
proof is completed by applying Lemma~\ref{symmetry}.
\end{proof}

\begin{theorem}\label{maintheorem}
The set $S$ contains
\begin{equation} \label{the_set}
\bigcup_{n\in\Z_+,\;
r\in \{0,...,2^{n+2}\}}\;\rme^{\pi i r/2^{n+1}}\,[0,2^{1/2^n}] .
\end{equation}
Hence $S$ contains the unit disc in $\C$.
\end{theorem}

\begin{proof}
For $n=0$ the theorem states that
\[
[0,2]\times\{ 1,i,-1,-i\}\subset S.
\]
This follows by combining Lemma~\ref{symmetry} with direct
calculations of $\Spec(A_c)$ when $c_n=1$ for all $n\in\Z$ (in which case $\Spec(A_c)=[-2,2]$) and when
$c_n=-1$ for all $n\in\Z$ (in which case $\Spec(A_c)=i[-2,2]$). For larger $n$ the first statement of the
theorem follows by applying Lemma~\ref{squareroot} inductively.
The
second statement is now a consequence of the fact that the set \eqref{the_set}
is dense in the unit disc.
\end{proof}

\section{The maps $\Gam_\pm$} \label{sec:maps}

A crucial role has been played in the proofs above by the nonlinear map $\Gamma_+$ on $\Ome$ introduced in Lemma \ref{square}, and this map will be key to the arguments that we make throughout this paper. And in fact a sequence which is almost a fixed point of $\Gamma_+$ (in a sense made precise below Lemma \ref{spaceinversion}) is central to the proof of Theorem \ref{maintheorem} in \cite{CCL}, though the proof is quite different and no mapping $\Gamma_+$ appears in \cite{CCL}.

The relationship between the above proof of
Theorem~\ref{maintheorem} and that in \cite{CCL} is clarified to
some extent by the following. Building on the definition of $\Gamma_+$ made above, let us define maps $\Gam_\pm:\Ome\to\Ome$
by $\Gam_\pm (b)=c$ where
\begin{equation}
c_0=\pm 1, \hspace{2em} c_{2n}+c_{2n+1}=0, \hspace{2em}
c_{2n}c_{2n-1}=b_n,\label{cfixed0}
\end{equation}
for all $n\in\Z$. We also define the space inversion symmetry $b\to
\wdh{b}$ by $\wdh{b}_n=b_{1-n}$ for all $n\in\Z$.

\begin{lemma}\label{spaceinversion}
If $\Gam_\pm (b)=c$ then $\Gam_\mp (\wdh{b})=\wdh{c}$. In particular
$\Gam_\pm (c)=c$ if and only if $\Gam_\mp (\wdh{c})=\wdh{c}$. Each
of the equations $\Gam_\pm (c)=c$ has exactly one solution.
\end{lemma}

\begin{proof}
Let $c=\Gam_+(b)$ and $d=\Gam_-(\wdh{b})$. Then $d_0=-1$,
$d_{2n}+d_{2n+1}=0$ and $d_{2n}d_{2n-1}=\wdh{b}_n=b_{1-n}$ for all
$n\in\Z$. Therefore $\wdh{d}_0=d_1=1$. Also
\[
\wdh{d}_{2n+1}+\wdh{d}_{2n}=d_{1-(2n+1)}+d_{1-2n}=d_{-2n}+d_{1-2n}=0
\]
and
\[
\wdh{d}_{2n}\wdh{d}_{2n-1}=d_{1-2n}d_{1-(2n-1)}%
=d_{2(1-n)-1}d_{2(1-n)}=\wdh{b}_{1-n}=b_n
\]
for all $n\in\Z$. Therefore $\wdh{d}=\Gam_+(b)=c$ and $d=\wdh{c}$.

The proof that $c=\Gam_-(b)$ implies $d=\Gam_+(\wdh{b})$ is similar.
The other statements of the lemma follow immediately.
\end{proof}

This paper and \cite{CCL} use three different special sequences. The
sequences $c_\pm$ are defined by $\Gam_\pm (c_\pm)=c_\pm$. It
follows directly from their definitions that $c_{+,0}=1$ and
$c_{+,1}=-1$ while $c_{-,0}=-1$ and $c_{-,1}=1$. However
\[
c_{+,n}=c_{-,n}=c_{+,1-n}=c_{-,1-n}
\]
for all $n\not= 0,\, 1$. The paper \cite{CCL} uses the sequence
$c_e$ such that $c_{e,0}=c_{e,1}=1$, while $c_{e,n}=c_{\pm,n}$ for
all other $n$. Because of the space inversion symmetry the use of
$c_+$ or $c_-$ in any proof is really a matter of convenience.

We now turn to the solution of the equation $A_c u=\lam u$ where
$u:\Z\to\C$ is an arbitrary sequence. The eigenvalue equation is
equivalent to the second order recurrence equation
\[
u_{n+1}+c_nu_{n-1}=\lam u_n.
\]

\begin{lemma}\label{uflip}
Suppose that $c\in\Ome$ and $\widehat{c}_n=c_{1-n}$ for all
$n\in\Z$; that $u_{n+1}+c_nu_{n-1}=\lam u_n$ for some $\lam\in\C$
and all $n\in\Z$ and $u_0=0$, $u_1=1$; and that
$\widehat{u}_{n+1}+\widehat{c}_n\widehat{u}_{n-1}=\lam
\widehat{u}_n$ for all $n\in\Z$ and $\widehat{u}_0=0$,
$\widehat{u}_1=1$. Then $|\widehat{u}_{n}|=|u_{-n}|$ for all
$n\in\Z$. In particular $u_n$ is bounded as $n\to\infty$ if and only
if $\widehat{u}_n$ is bounded as $n\to -\infty$.
\end{lemma}

\begin{proof}
If one puts $v_n=u_{-n}$ then
\begin{equation}
\widehat{c}_{n+1}v_{n+1}+v_{n-1}=c_{-n}u_{-n-1}+u_{-n+1}%
=\lam u_{-n}=\lam v_n \label{growthlemma}
\end{equation}
for all $n\in\Z$. Define $a:\Z\to\{\pm 1\}$ by $a_0=1$ and
$a_n/a_{n-1}=\widehat{c}_n$ for all $n\in\Z$. If one now puts
$w_n=a_n v_n$ for all $n\in\Z$ then (\ref{growthlemma}) implies
\[
w_{n+1}+\widehat{c}_nw_{n-1}=\lam w_n
\]
for all $n\in\Z$. Since $w_0=v_0=u_0=0$ it follows that there exist
$\gam$ such that $w_n=\gam \widehat{c}_n$ for all $n\geq 1$. But
$|c_n|=1$ for all $n$, so one obtains $|\gam|=1$ by evaluating this
identity for $n=1$. Therefore $|\widehat{u}_n|=|w_n|=|v_n|=|u_{-n}|$
for all $n\geq 1$.
\end{proof}

\begin{corollary}\label{uevenodd}
Let $u_+,\, u_-,\, u_e:\Z\to \C$ be the solutions of
$u_{n+1}+c_nu_{n-1}=\lam u_n$ for all $n\in\Z$ subject to $u_0=0$
and $u_1=1$, if $c$ is put equal to $c_+,\, c_-,\, c_e$
respectively. Then $u_{+,n}=u_{e,n}$ for all $n\in\Z$. Moreover
$|u_{-,n}|=| u_{e,n}|$ for all $n\in\Z$.
\end{corollary}

\begin{proof}
The first statement is proved by an elementary computation. For the
second we use $c_e=\widehat{c_e}$ and $c_-=\widehat{c_+}$.
Lemma~\ref{uflip} now yields
\[
|u_{-,n}|=|u_{+,-n}|=|u_{e,-n}|=|u_{e,n}|
\]
for all $n\in\Z$.
\end{proof}

The main step in the proof of Theorem \ref{maintheorem} in \cite{CCL} is contained in the following proposition (we quote here the parts of
\cite[Proposition~2.1]{CCL} which we use immediately or later in section \ref{secondproofsection}).

\begin{proposition}\label{CCLprop}
Let $u_e$ be defined as in Corollary \ref{uevenodd} and define $p_{i,j}\in \Z$ for $i,j\in\N$ by the formula
\[
u_{e,i}=\sum_{j=1}^i p_{i,j}\lam^{j-1}
\]
with $p_{i,j}=0$ if $j>i$. Let $Y$ denote the set of $(i,j)\in \N^2$ such that $p_{i,j}\neq 0$. Then $p_{i,j}\in\{0,1,-1\}$ for all $i,\, j$ and $(i,j)\in Y$ if and only if one of the following holds.\\
(1) $i=j=1$;\\
(2) $i$ and $j$ are both even and $(i/2,j/2)\in Y$;\\
(3) $i$ and $j$ are both odd and $((i+1)/2,(j+1)/2)\in Y$;\\
(4) $i$ and $j$ are both odd and $((i-1)/2,(j+1)/2)\in Y$.
\end{proposition}

The following result is an immediate corollary of this proposition and Corollary \ref{uevenodd}, which together imply that $|u_{e,i}|\leq (1-|\lambda|)^{-1}$ for $i\in \Z$ and $|\lambda|<1$.

\begin{theorem} \cite{CCL} \label{CCLprop3}
As in Corollary \ref{uevenodd}, let $u_e:\Z\to \C$ be the solution of
$u_{n+1}+c_nu_{n-1}=\lam u_n$ for all $n\in\Z$ subject to $u_0=0$
and $u_1=1$, with $c=c_e$. Then $u_e\in \ell^\infty(\Z)$ for $|\lam|<1$, so that $\Spec(A_{c_e})$ contains the unit disc.
\end{theorem}

Since Corollary \ref{uevenodd} has shown that $u_{+}=u_{e}$, it is clear from Theorem \ref{CCLprop3} that $\Spec(A_{c_+})$ also contains the unit disc. In fact this is precisely its spectrum.

\begin{theorem}\label{c+spectrum}
If $c_+$ is the unique solution of $\Gam_+(c)=c$ then
\[
\Spec(A_{c_+})=\{ z:|z|\leq 1\}.
\]
\end{theorem}

\begin{proof}
It remains only to show that $\Spec(A_{c_+})\subset\{ z:|z|\leq 1\}$. If $\lam\in \Spec(A_{c_+})$ then repeated applications of the first
part of Lemma~\ref{squareroot} yield $\lam^{2^n}\in \Spec(A_{c_+})$
for all $n\geq 1$. Since the spectrum is a bounded set, it follows
that $|\lam|\leq 1$.
%The proof that $\Spec(A_{c_+})\supseteq\{ z:|z|\leq 1\}$ uses the
%method of \cite{CCL}. That paper uses $c_e$, but this is unimportant
%because $u_e=u_+$ by Corollary~\ref{uevenodd}.
\end{proof}

We will (rather arbitrarily) focus on the mapping $\Gamma_+$ rather than $\Gamma_-$ in the remainder of the paper. The following lemma, which shows   that the set of periodic sequences
is invariant under the action of
$\Gam_{+}$, will play a key role.

\begin{lemma} \label{lem_periodic} If $b\in \Omega$ is periodic with
period $N$, i.e.\ $b_{n+N} = b_n$, $n\in\Z$,
then $c=\Gam_{+}(b)$ is $4N$-periodic.
Conversely, if $b\in\Omega$, $c=\Gam_{+}(b)$, and $c$
is $2N$-periodic for some $N\in\N$, then $b$ is $N$-periodic.
\end{lemma}

\begin{proof}
First note that, if $c=\Gam_{+}(b)$ and one defines $\tilde c\in
\Omega$ by $\tilde c_n=c_{2n}$, $n\in\Z$, then
\begin{equation}
c=\Gam_{+}(b) \Leftrightarrow (\tilde c_0=1, \quad \tilde c_n =
-b_n \tilde c_{n-1}, \; c_{2n+1} = - \tilde c_n, \;
n\in\Z).\label{equiv}
\end{equation}
Therefore
\begin{equation}
\tilde c_{m+n} = \tilde c_m\,(-1)^n\prod_{j=1}^n
b_{m+j}\label{equiv2}
\end{equation}
 for all $m\in\Z$ and $n\in\N$. If $b$ is $N$-periodic, then
\[
\tilde c_{m+2N} = \tilde c_m\,\prod_{j=1}^{2N} b_{m+j} =
\tilde c_m\,\prod_{j=1}^{N} b_{m+j}^2 = \tilde c_m,
\]
for all $m\in\Z$. Therefore $c$ is $4N$-periodic.

Conversely, if $c=\Gam_{+}(b)$, for some $b\in
\Omega$, and $c$ is $2N$-periodic for some $N\in\N$, then
$\tilde c$ is $N$-periodic and, from \eqref{equiv}, it follows that
$b$ is $N$-periodic.
\end{proof}

To illustrate the above lemma, define $c^-,c^+ \in \Ome$ by $c^-_n =  -1$, $c^+_n=1$, for $n\in\Z$, and define the sequences $c^{(m,+)}, c^{(m,-)}\in \Ome$, for $m=0,1,...$, by
\begin{equation} \label{periodic_sequences}
c^{(0,\pm)} = c^\pm, \quad c^{(m,\pm)} = \Gamma_+(c^{(m-1,\pm)}), \quad m\in \N.
\end{equation}
Then explicit calculations of the action of $\Gamma_+$ yield that
$c^{(1,+)}=\Gam_{+}(c^+)$ is $4$-periodic (but not periodic with any
smaller period), with $c^{(1,+)}_{-1}=c^{(1,+)}_0=1$, $c^{(1,+)}_1=c^{(1,+)}_2=-1$. On the other hand, $c^{(1,-)}=\Gam_{+}(c^-)$ is $2$-periodic  (and so also
$4$-periodic), with $c^{(1,-)}_n = (-1)^n$ for all $n\in\Z$.

Both these calculations, of course, are consistent with the lemma, which implies that $c^{(m,\pm)}$ is $N$-periodic with $N=4^m$, so that, using the notation \eqref{eq:piinf} (dropping $\sig$ given that $\sig=1$),
\begin{equation} \label{inclusion}
\Spec(A_{c^{(m,\pm)}}) \subset \pi_{4^m}, \quad m=0,1,... .
\end{equation}

Although we do not have an explicit formula for the sequences $c^{(m,\pm)}$, it is easy to compute $\Spec(A_{c^{(m,\pm)}})$. By Lemma \ref{squareroot}, if $c=\Gamma_+(b)$, then
\begin{equation} \label{speciter}
\Spec(A_c) = \{\pm \sqrt{\lambda}: \lambda \in \Spec(A_b)\}.
\end{equation}
The proof of Theorem \ref{maintheorem} begins with the observation that $\Spec(A_{c^+})=[-2,2]$ and $\Spec(A_{c^-})=i[-2,2]$. Combining this observation with \eqref{speciter} we easily prove by induction that
\begin{equation} \label{spec_cplus}
\Spec(A_{c^{(m,+)}}) = \left\{r\,\rme^{\pi i j/2^{m}}: 0\leq r \leq 2^{1/2^m}, \;j\in \{0,...,2^{m+1}-1\}\right\}
\end{equation}
and
\begin{equation} \label{spec_cminus}
\Spec(A_{c^{(m,-)}}) =\rme^{\pi i/2^{m+1}}\,\Spec(A_{c^{(m,+)}}).
\end{equation}

Combining equations \eqref{inclusion}, \eqref{spec_cplus} and \eqref{spec_cminus}, we see that we have shown that
$$
\left\{r\,\rme^{\pi i j/2^{m}}: 0\leq r \leq 2^{1/2^{m+1}}, \;j\in \{0,...,2^{m+2}-1\}\right\}\subset \pi_{4^m}, \quad m=0,1,... .
$$
Thus we have shown the following modification of Theorem \ref{maintheorem} which, of course, by \eqref{eq:pidef}, has Theorem \ref{maintheorem} as a corollary.

\begin{theorem} \label{denseness}
The set $\pi_\infty$ contains the set \eqref{the_set}, and so is dense in the unit disc in $\C$.
\end{theorem}

We know $\Spec(A_{c^{(m,\pm)}})$ explicitly, but do not have explicit formulae for the sequences $c^{(m,\pm)}$. However we can show that $c^{(m,\pm)}$ converges pointwise to the sequence $c_+$, the unique fixed point of $\Gamma_+$, as $m\to \infty$. This is the content of the next two lemmas. We omit a proof of the first of these lemmas which is an easy consequence, by simple induction arguments, of the definition of $\Gam_+$.

\begin{lemma} \label{maplem}
 If $b\in\Omega$ and $c=\Gamma_+(b)$, then $c_0=c_{+,0}$ and $c_1 =c_{+,1}$. If, for some $N\in\N$, $b_m=c_{+,m}$ for $m = 1,...,N$, then also $c_m = c_{+,m}$ for $m = 2,...,2N+1$. If, for some $N\in\Z_+$, $b_{-m}=c_{+,-m}$ for $m = 0,1,...,N$, then $b_{-m}=c_{+,-m}$ for $m = 1,2,...,2N+2$.
\end{lemma}

\begin{lemma} \label{strong_converg}
Let $b \in\Omega$, and define $c^{(n)}\in \Ome$ for $n\in \N$ by $c^{(1)}=\Gam_+(b)$ and $c^{(n+1)} = \Gam_+(c^{(n)})$, $n\in\N$. Then,  for $n\in\N$,
$$
c^{(n)}_m = c_{+,m}, \quad m = 2-2^n,3-2^n,...,2^n-1,
$$
so that $c^{(n)}\to c_+$ pointwise and $A_{c^{(n)}}$ converges strongly to $A_{c_+}$ as $n\to\infty$. Further,
$$
\Spec(A_{c^{(n)}}) \subset \{\lambda: |\lambda|\leq 2^{1/2^n}\}.
$$
\end{lemma}
\begin{proof} The first equation follows by induction from Lemma \ref{maplem}. The second equation follows by induction from \eqref{speciter} and the trivial bound that $\Spec(A_b)\subset \{\lambda:|\lambda|\leq 2\}$, which holds for all $b\in \Ome$.
\end{proof}

\section{The mapping $\Gamma_{\sig,+}$}\label{related}

For the rest of the paper we consider operators $A_c$ for which the
coefficients $c_n$ take values in $\{\pm \sig\}$, where $0<\sig \leq
1$; that is, in the notation we have introduced in the introduction, we assume that $c\in \Omega_\sig$, for some $\sig\in (0,1]$.

The mapping $\Gamma_+$ that we have introduced continues to play an important role. We extend the mapping so that it operates on $\Omega_{\sig^2}$, defining, for $\sig\in (0,1]$, $\Gamma_{\sig,+}:\Omega_{\sig^2}\to \Omega_\sig$ by
\begin{equation} \label{def_Gam_sig_plus}
\Gamma_{\sig,+}(c) = \sig \Gamma_+(\sig^{-2}c).
\end{equation}
In other words, for $b\in \Omega_{\sig^2}$, $c=\Gamma_{\sig,+}(b)$ is the unique sequence in $\Omega_\sig$ satisfying
\begin{equation}
c_0=\sig, \hspace{2em} c_{2n}+c_{2n+1}=0, \hspace{2em}
c_{2n}c_{2n-1}=b_n.\label{iteration2a}
\end{equation}

Main properties of the mapping $\Gamma_{\sig,+}$ for our purposes are contained in the following extension of Lemma \ref{squareroot}. We will need to refer to a
number of circular annuli, and use in this lemma and subsequently the notation
\begin{equation} \label{annuli}
\bkt{  a,b } =\{ \lam:a\leq |\lam |\leq b\}.
\end{equation}

\begin{lemma}\label{lemmaB}If $b\in\Ome_{\sig^2}$ and $c=\Gam_{\sig,+}(b)\in \Ome_\sig$, then
\begin{equation}
(\,\lam\in \Spec(A_b) \,)\Leftrightarrow (\, \pm \sqrt{\lam} \in
\Spec(A_c) \,).\label{sigmaroot0}
\end{equation}
Hence
\begin{equation}
(\,\lam\in S_{\sig^2}\,)\implies (\, \pm \sqrt{\lam} \in
S_\sig\,)\label{sigmaroot}
\end{equation}
and
\begin{equation}
(\, \bkt{a,b}\subseteq S_{\sig^2}\,)\implies%
 ( \, \bkt{a^{1/2},b^{1/2}}\subseteq S_{\sig}\,).\label{sigroot2}
\end{equation}
\end{lemma}

\begin{proof}
We modify the calculations in Section~\ref{main}.
Lemma~\ref{symmetry} is valid as it stands. In Lemma~\ref{square} we
assume that $b\in \Ome_{\sig^2}$, and define $c\in \Ome_\sig$ by $c=\Gam_{\sig,+}(b)$, and apply Case 2 of Theorem \ref{spectralequality} in place of Case 1.
This leads to the conclusion $\Spec(A_c^2)= \Spec(A_b)$ as in
Lemma~\ref{squareroot}. (\ref{sigmaroot}) and (\ref{sigroot2})
follow by choosing $b\in\cE_{\sig^2}$ and using
Proposition~\ref{quasiergodic}.
\end{proof}

\section{Periodic and paired periodic operators} \label{periodic_ops}

To prove our main theorem  we need results on operators $A_c$ on
$\ell^2(\Z)$ that have one periodic structure for $n\geq 0$ and
another for $n<0$ (which we term paired periodic operators). The essential spectrum of such an operator is the
union of the essential spectra of the periodic operators involved,
which may be calculated explicitly using their Bloch decompositions.

There may also be substantial inessential spectrum, in particular, open subsets of the spectrum where $A_c-\lam I$ is Fredholm but has non-zero index. These parts of the spectrum (and the corresponding values of the index) can be computed by application of general results for block Toeplitz operators, which have been developed to a high degree of sophistication; see \cite{B2,BS,BK,BS2} and the references therein. We need only a small part of this theory, and it is easy to develop this from first principles. We do this in a short Lemma \ref{lem_two_periods} below, inspired by earlier analysis in \cite{DS, EBD0,EBD1}, and particularly
\cite[Theorem~12]{EBD1}.
%so that The analysis
%of such operators is very similar to that of Toeplitz and
%Wiener-Hopf  operators and has been developed to a high degree of
%sophistication in a variety of technical contexts; see
%\cite{B1,B2,BS,EBD2000,EBD0,DS,Re,RT}, and particularly
%\cite[Theorem~12]{EBD1}. We only need a very small part of this
%theory, and it is easy to develop this from first principles, which we do in Lemma \ref{lem_two_periods} below.
 Both the proof of Lemma \ref{lem_two_periods}, and the effective application of this lemma to prove Theorem \ref{maintheorem3}, depend on the next two lemmas which describe properties of the spectra and eigenfunctions of periodic operators.

We assume throughout this section that the parameter $\sig\in (0,1)$.

\begin{lemma}\label{quadraticbounds}
Let
\begin{equation}
\Phi(\tau,\gam)= \frac{\Re(\tau)^2}{(1+\gam)^2}+
\frac{\Im(\tau)^2}{(1-\gam)^2}\label{Phidef}
\end{equation}
where $\tau\in\C$ and $-1< \gam <1$. Then the quadratic equation
\begin{equation}
z^2-\tau z+\gam=0\label{zeq}
\end{equation}
has a solution satisfying $|z|=1$ if and only if $\Phi=1$. If
$\Phi<1$ then both solutions satisfy $|z|<1$. If $\Phi>1$ then one
solution satisfies $|z|<1$ and the other satisfies $|z|>1$.
\end{lemma}

\begin{proof}
For $\theta\in\R$,  $z=\rme^{i\theta}$ is a solution of (\ref{zeq}) if and only if
\[ \cos(\theta)=\frac{\Re(\tau)}{1+\gam}, \hspace{2em}
\sin(\theta)=\frac{\Im(\tau)}{1-\gam},
\]
so that \eqref{zeq} has a solution satisfying $|z|=1$ if and only if $\Phi(\tau,\gamma)=1$.

The set $U=\{\tau\in\C:\Phi<1\}$ is connected and contains the
origin. Since the solutions of (\ref{zeq}) depend continuously on
$\tau$, and both solutions satisfy $|z|<1$ if $\tau=0$, it follows
that both satisfy $|z|<1$ for all $\tau\in U$. The case $\Phi>1$ is
similar.
\end{proof}

The following lemma is closely related to a similar result for the non-self-adjoint Anderson model in \cite[Theorem~11]{EBD1}.

\begin{lemma}\label{BIO}
If $c\in \Ome_\sig$ and $\lam\in \C$ then the space of all
solutions of $A_cf=\lam f$ is two-dimensional. If $c$ is periodic
with period $p$ then the asymptotic behaviour as $n\to\pm\infty$ of
the solutions is determined by the solutions $z_1,\, z_2$ of the
polynomial $z^2-\tau(\lam) z+\gam=0$, where
$\tau(\lam)$ is a monic polynomial in $\lam$ with degree $p$, given by $\tau(\lambda) = \tr(T_p)$, where $T_p=X_p X_{p-1}\ldots X_1$ and
$$
X_n = \mtrx{0&1\\ -c_n&\lam},
$$
and $\gam=\det(T_p)=\pm \sig^p$.
Ordering the two solutions so that $|z_1|\geq |z_2|$, there are three cases:
\begin{enumerate}
\item
$\lambda$ lies in the closed set
\[
B_c=\{\lam: |z_1|=1\mbox{ and } |z_2|=\sig^p\}.
\]
This set is the spectrum of $A_c$, equivalently, the set of $\lam$ for
which $A_cf=\lam f$ has a bounded solution.
\item
$\lam$ lies in the open set
\[
I_c=\{\lam: 1> |z_1|\geq|z_2|>\sig^p\}.
\]
This is the case if and only if all solutions of $A_cf=\lam f$ decay exponentially as
$n\to +\infty$.
\item
$\lam$ lies in the open set
\[
O_c=\{\lam: |z_1|>1\mbox{ and } |z_2|<\sig^p\}.
\]
This is the case if and only if there exists a solution of $A_cf=\lam f$ that decays
exponentially as $n\to +\infty$ and grows exponentially as $n\to
-\infty$, and another solution that decays exponentially as $n\to
-\infty$ and grows exponentially as $n\to +\infty$.
\end{enumerate}
\end{lemma}

\begin{proof}
The sequence $f:\Z\to\C$ is a solution of $A_cf=\lam f$ if and only
if $f_{n+1}+c_n f_{n-1}=\lam f_n$ for all $n\in\Z$. This recurrence
relation can be rewritten in the form
\begin{eqnarray*}
\vctr{ f_n\\ f_{n+1}}&=&\mtrx{0&1\\ -c_n&\lam}\vctr{ f_{n-1}\\ f_n}\\
&=& X_n\vctr{ f_{n-1}\\ f_n}\\
&=& T_n\vctr{ f_{0}\\ f_1}
\end{eqnarray*}
where $T_n=X_n X_{n-1}\ldots X_1$. If $c$ is periodic with period
$p$, then the asymptotic behaviour of the two-dimensional space of
eigenfunctions $f$ is determined by the magnitude of the eigenvalues
$z_1,\, z_2$ of $T_p$. These are the solutions of the equation
$z^2-\tau z+\gam=0$ where $\tau=\tr(T_p)$ and $\gam=\det(T_p)$.
A simple induction establishes that the $(i,j)$-th entry of $T_p$ is a polynomial in $\lam$ with degree less than $p$ unless $i=j=2$ in which case it is a monic polynomial with degree $p$. Therefore $\tau$ is a monic polynomial in $\lam$ with degree $p$. However
\[
\det(T_p)=\prod_{r=1}^p \det(X_r)= c_1 \dots c_p = \pm \sig^p
\]
does not depend on $\lam$. The continuous dependence of the roots of a polynomial on its
coefficients implies that $B_c$ is closed while $I_c$ and $O_c$ are
open. An application of Lemma~\ref{quadraticbounds} now completes
the proof. One sees, in particular, that
\[
\Spec(A_c)= B_c=\{\lam: \Phi(\tau,\gam)=1 \}.
\]
\end{proof}

Our next lemma enables us to determine the sets $I_c$ and $O_c$ for certain important periodic sequences $c$, and to determine the spectra of certain paired periodic operators. We continue with the assumptions and notation of Lemma~\ref{BIO}.

\begin{lemma}\label{findIcOc}
If $V$ is a connected component of $\C\backslash B_c$ then $V\subseteq I_c$ or $V\subseteq O_c$. If $V$ is unbounded then $V\subseteq O_c$, and if $0\in V$ then $V\subseteq I_c$. If $\C\backslash B_c$ has exactly two components then the bounded component equals $I_c$ and the unbounded component equals $O_c$.
\end{lemma}

\begin{proof}
We first observe that $V$, $I_c$ and $O_c$ are all open sets and that their definitions imply directly that $I_c,\, O_c$ are disjoint. Therefore $V=(V\cap I_c)\cup (V\cap O_c)$, where the two intersections on the right-hand side are disjoint. Since $V$ is connected, it follows that
$V=V\cap I_c$ or $V=V\cap O_c$ This completes the proof of the first statement.

Lemma~\ref{BIO}~case~1 implies that
\[
B_\sig=\Spec(A_c)\subseteq \{\lam: |\lam|\leq 1+\sig\}.
\]
Therefore $\C\backslash B_c$ has only one unbounded component $V$ and it contains $\{ \lam : |\lam |>1+\sig\}$. To prove that $V\subseteq O_c$ it is sufficient by the first part of this proof to find a single point $\lam\in V\cap O_c$. The fact that $\tau$ is a polynomial with degree $p$ implies that $|\tau(\lam)|\to\infty$ as $|\lam|\to\infty$. This implies that the solutions of $z^2-\tau(\lam)z+\gam=0$, where $\gam=\pm \sig^p$, are $z\sim \tau(\lam)$ and $z\sim \gam/\tau(\lam)$ to leading order for all large enough $|\lam|$. Therefore $\lam\in O_c$ for all such $\lam$.

The proof is completed by proving that $0\in I_c$. For $\lam=0$ one has $T_p=X_pX_{p-1}\ldots X_1$ where each $X_r$ is of the form $\txtmtrx{0&1\\ \pm \sig&0}$. If $p=2m$ it follows that $T_p=\txtmtrx{\pm \sig^m&0\\0&\pm \sig^m}$. The fundamental equation must therefore take one of the forms $z^2-2\sig^m z+\sig^{2m}=0$, $z^2+2\sig^m z+\sig^{2m}=0$ or $z^2-\sig^{2m}=0$. In each case both solutions have modulus $\sig^{p/2}<1$. The same holds if $p=2m+1$.

The final statement of the lemma follows from the following observations. There must be a component of $\C\setminus B_c$ that contains $0$ and there must be an unbounded component. The first part of the  proof shows that these are distinct, and the extra hypothesis is that there are no other components.
\end{proof}

Our next task is to determine the sets $B_c,\, I_c$ and $O_c$ for
certain particular periodic sequences.

\begin{lemma}\label{c1}
If $c_n=\sig$ for all $n\in\Z$ then $\Spec(A_c)$ is the ellipse
\begin{eqnarray}
\Spec(A_c)&=&\left\{ u+iv: \frac{u^2}{(1+\sig)^2}+\frac{v^2}{(1-\sig)^2}=1\right\}\label{c1Eucl}\\
&=&\left\{ \rho\rme^{i\theta}:
\rho=\frac{1-\sig^2}{\sqrt{1+\sig^2-2\sig
\cos(2\theta)}}\right\}.\label{c1polar}
\end{eqnarray}
Moreover the interior $U$ of the ellipse equals $I_c$ and the
exterior $V$ of the ellipse equals $O_c$.
\end{lemma}

\begin{proof}
We have $p=1$ and $T_1=\txtmtrx{0&1\\ -\sig&\lam}$, so
$\tau(\lam)=\lam$ and $\gam=\sig$. Using (\ref{Phidef}) we deduce that
$\Spec(A_c)$ is given by (\ref{c1Eucl}). The proof is completed by using Lemma~\ref{findIcOc}.
\end{proof}

\begin{lemma}\label{c2}
If $c_n=-\sig$ for all $n\in\Z$ then $\Spec(A_c)$ is the ellipse
\begin{eqnarray}
\Spec(A_c)&=&\left\{ u+iv: \frac{u^2}{(1-\sig)^2}+\frac{v^2}{(1+\sig)^2}=1\right\}\label{c2Eucl}\\
&=&\left\{ \rho\rme^{i\theta}:
\rho=\frac{1-\sig^2}{\sqrt{1+\sig^2+2\sig
\cos(2\theta)}}\right\}.\label{c2polar}
\end{eqnarray}
Moreover $I_c$ is the interior of the ellipse and $O_c$ is the
exterior of the ellipse.
\end{lemma}

\begin{proof}
We have $p=1$ and $T_1=\txtmtrx{0&1\\ \sig&\lam}$, so
$\tau(\lam)=\lam$ and $\gam=-\sig$. We omit the rest of proof, which is almost identical to that of Lemma~\ref{c1}.
\end{proof}

In the following lemma, starting from Lemmas \ref{c1} and \ref{c2}, and making successive applications of Lemma \ref{lemmaB}, we compute the spectra of a family of periodic sequences, namely the sequences $c^\pm = \sig c^{(n,\pm)}\in \Ome_\sig$, defined by \eqref{periodic_sequences}. By Lemma \ref{lem_periodic}, these sequences are periodic of period $\leq 4^n$.

This next lemma applies for $0<\sig<1$. The corresponding result for $\sig=1$ is equations \eqref{spec_cplus} and \eqref{spec_cminus} above.

\begin{lemma} \label{lem_sf} Suppose $n\in \Z_+$ and $c^+ = \sig c^{(n,+)}$, $c^- = \sig c^{(n,-)}$. Then
\begin{equation} \label{spectral_formula}
\Spec( A_{c^\pm}) = \{\,\rho\rme^{\ri\theta}: \rho = \rho_n^\pm(\theta,\sigma)\}
\end{equation}
where
$$
\rho_0^+(\theta,\sigma) = \frac{1-\sigma^2}{\left(1+\sigma^2 - 2\sigma\cos 2\theta\right)^{1/2}}, \quad \rho_0^-(\theta,\sigma) = \frac{1-\sigma^2}{\left(1+\sigma^2 + 2\sigma\cos 2\theta\right)^{1/2}},
$$
and, for $n\in \N$,
$$
\rho_n^\pm(\theta,\sigma) =\left(\rho_0^\pm(2^n\theta,\sigma^{2^n})\right)^{1/2^n} = \frac{\left(1-\sigma^{2^{n+1}}\right)^{1/2^n}}{\left(1+\sigma^{2^{n+1}} \mp 2\sigma^{2^n}\cos \left(2^{n+1}\theta\right)\right)^{1/2^{n+1}}}.
$$
Moreover,
$$
I_{c^\pm} = \left\{\,\rho \rme^{i\theta}: 0\leq \rho < \rho_n^\pm(\theta,\sigma) \right\}
$$
and
$$
O_{c^\pm} = \left\{\,\rho \rme^{i\theta}: \rho > \rho_n^\pm(\theta,\sigma) \right\}.
$$
\end{lemma}

\begin{proof} Our proof of \eqref{spectral_formula} is by induction. We note first that \eqref{spectral_formula} holds for $n=0$ by Lemmas \ref{c1} and \ref{c2}. Suppose now that \eqref{spectral_formula} holds for some $n\geq 0$ and all $0<\sig<1$. Then
$$
\Spec(A_{\sig^2 c^{(n,\pm)}}) = \{\rho\rme^{\ri\theta}: \rho=\rho_n^\pm(\theta,\sigma^2)\} = \left\{\rho\rme^{\ri\theta}:\rho=\left( \rho_0^\pm(2^n\theta,\sigma^{2^{n+1}})\right)^{1/2^n}\right\}.
$$
Further, since $\sig c^{(n+1,\pm)} = \sig \Gamma_+(c^{(n,+)})=\Gamma_{\sig,+}(\sigma^2 c^{(n,+)})$, it follows from Lemma \ref{lemmaB} that
$$
\Spec(A_{\sig c^{(n+1,\pm)}}) = \left\{\pm \sqrt{\lambda}: \lambda\in \Spec(A_{\sig^2 c^{(n,\pm)}})\right\}.
$$
Combining these equations, we see that \eqref{spectral_formula} holds with $n$ replaced by $n+1$. Thus \eqref{spectral_formula} follows by induction.

The formulae for $I_{c^\pm}$ and $O_{c^\pm}$ follow from \eqref{spectral_formula} and Lemma \ref{findIcOc}.
\end{proof}

We remark that
$\rho_n^-(\theta,\sigma) =\rho_n^+(\theta\pm\pi/2^{n+1},\sigma)$, so that the spectra of $A_{c^\pm}$ in the above lemma are related by
$$
\Spec(A_{c^{+}}) = \rme^{\pm i\pi/2^{n+1}}\, \Spec(A_{c^-}).
$$
This is a symmetry which is surprising from an inspection of the sequences $c^\pm$, which need not even have the same period. (For example, as observed in Section \ref{sec:maps},  $c^+$ has period 4 and $c^-$ period 2 in the case $n=1$.)

In principle, since $c^\pm$ is periodic, \eqref{spectral_formula} should be computable alternatively from the characterisation of the spectrum for general periodic sequences in Lemma \ref{BIO}. As an example of this, for the sequence $c^-=\sig c^{(1,-)}$ which has period $2$, with $c^-_n = (-1)^n\sig$,  the transfer matrix $T_2$ is
given by
\[
T_2=X_2X_1=\mtrx{0&1\\ -\sig&\lam}\mtrx{0&1\\
\sig&\lam}=\mtrx{\sig&\lam\\ \sig\lam&-\sig+\lam^2}.
\]
Applying Lemmas~\ref{quadraticbounds} and \ref{BIO} with
$\tau=\lam^2$ and $\gam=-\sig^2$, we find that $\Spec(A_c)$ is the
set of all $\lam=u+iv$ for which
\[
\frac{(u^2-v^2)^2}{(1-\sig^2)^2}+ \frac{(2uv)^2}{(1+\sig^2)^2}=1.
\]
If one puts $\lam=\rho \rme^{i\theta}$, then this may be rewritten in the form \eqref{spectral_formula}.

The main point of the above theory and calculations are to prove and prepare the use of the following result on operators $A_c$ that are paired periodic operators. To state this result let us introduce the notations
\begin{equation}
E_\sig=\left\{ x+iy:\frac{x^2}{(1+\sig)^2}+
\frac{y^2}{(1-\sig)^2}< 1\right\}\label{ell1}
\end{equation}
and
\begin{equation}
E_{-\sig}=\left\{ x+iy:\frac{x^2}{(1-\sig)^2}+
\frac{y^2}{(1+\sig)^2}<1\right\}\label{ell2},
\end{equation}
so that $E_\sig$ and $E_{-\sig}$ are the interiors of the ellipses introduced in Lemmas~\ref{c1} and
\ref{c2}.

The following lemma is analogous to \cite[Theorem~12]{EBD1}, proved there for the non-self-adjoint Anderson model.

\begin{lemma} \label{lem_two_periods} Suppose that $c\in \Ome_\sig$ is periodic and $\tau\in \{\sig,-\sig\}$, and define $c^*\in \Ome_\sig$ by $c^*_n = c_n$, for $n\geq 0$, and $c^*_n=\tau$ for $n< 0$. Then
$$
\Spec(A_{c^*}) \supset \overline{I_c}\setminus E_\tau.
$$
\end{lemma}
\begin{proof}
Since $\Spec(A_{c^*})$ is closed it is enough to show that $\Spec(A_{c*}) \supset I_c\setminus \overline{E_\tau}$. So suppose that $\lambda \in I_c\setminus \overline{E_\tau}$. Then, by Lemmas \ref{c1} and \ref{c2} and Lemma \ref{BIO}, since $\lam\not\in \overline{E_\tau}$, it follows
that there exists a non-trivial solution $f$ of $A_{c^*}f=\lam f$ such that $f_n\to
0$ exponentially as $n\to -\infty$. Since $\lam\in I_c$, again applying Lemma \ref{BIO}, it follows that this
solution $f$ also decays exponentially as $n\to +\infty$. Thus $\lam$ is an eigenvalue of $A_{c^*}$ so $\lam\in \Spec(A_{c^*})$.
\end{proof}

\section{First proof of the main theorem}\label{firstproof}

This section is devoted to the proof of Theorem~\ref{maintheorem3},
in which we assume that $0<\sig<1$.

\begin{lemma}\label{lemmaA} We have
\begin{equation}
\Spec(A_c)\subseteq \bkt{1-\sig, 1+\sig} \label{ring}
\end{equation}
and
\begin{equation}
\Spec(A_c)\subseteq \{ x+iy:|x|+|y|\leq
\sqrt{2(1+\sig^2)}\},\label{numrange}
\end{equation}
for every choice of $c\in\Ome_\sig$.
\end{lemma}

\begin{proof}
We regard $V_cR$ as a small perturbation of $L$ in the identity
$A_c=V_cR+L$, noted in the proof of Lemma \ref{square}. Since $L$ is a unitary operator with spectrum
$\{z:|z|=1\}$, we have
\[
\norm (L-zI)^{-1}\norm =\left| 1-|z|\, \right|^{-1}
\]
for all $z$ not on the unit circle. The inclusion (\ref{ring}) now
follows from $\norm V_cR\norm=\sig$ by a perturbation argument; see
\cite[Th.~9.2.13]{LOTS}.

The inclusion (\ref{numrange}) depends on an estimate of the
numerical range of $A_c$. Following \cite[Section~9.3]{LOTS},
$x+iy\in\Num(A_c)$ if there exists $f\in \ell^2(\Z)$ such that
$\norm f\norm =1$ and $x+iy=\langle A_c f,f\rangle$. This implies
that
\[
x=\frac{1}{2}\langle (A_c+A_c^\ast)f,f\rangle, \hspace{2em} %
y=-\frac{i}{2}\langle (A_c-A_c^\ast)f,f\rangle.
\]
Therefore
\[
x+y=\frac{1}{2}\langle Bf,f\rangle
\]
where
\[
B=(A_c+A_c^\ast)-i(A_c-A_c^\ast).
\]
A simple calculation shows that $B_{m,n}=0$ unless $|m-n|=1$, while
\[
B_{n,n+1}=\overline{ B_{n+1,n}}=(1\pm \sig)-i(1\mp \sig).
\]
Therefore $|B_{n+1,n}|=|B_{n,n+1}|=\sqrt{2(1+\sig^2)}$ for all $n\in\Z$ and
\[
x+y \leq \frac{1}{2}\norm B\norm\leq \sqrt{2(1+\sig^2)}.
\]
The other three steps in the proof of the bound for $|x|+|y|$ are
similar.
\end{proof}

The statement of our main theorem refers to the open set
\begin{equation}
H_\sig=E_\sig\cap E_{-\sig},\label{Hsigma}
\end{equation}
the intersection of the ellipses $E_\sig$ and $E_{-\sig}$. This set
satisfies
\begin{equation}
\bkt{ 0,1-\sig} \subseteq \overline{H_\sig}\subseteq \bkt{0,r_{\sig}}
\label{Hsigsize}
\end{equation}
where
\begin{equation}
r_{\sig}=\frac{  1-\sig^{2}  }%
{  \sqrt{1+\sig^{2}\,}  }.\label{rndef}
\end{equation}

\begin{theorem}\label{maintheorem3}
If $0<\sig<1$ then
\[
\{ \lam:|\lam|\leq 1\}\backslash H_\sig \subseteq S_\sig.
\]
\end{theorem}

\begin{proof} Note first that if $c^\pm$ and $\rho_n^\pm(\theta,\sig)$ are defined as in Lemma \ref{lem_sf}, then
$$
\rho_n^\pm(\theta, \sig) \geq \rho_{\sig,n} = \left(\frac{1-\sig^{2^{n+1}}}{1+\sig^{2^n}}\right)^{1/2^n},
$$
for all $\theta\in \R$, so that $I_{c^\pm} \supset  \{\lam:|\lam|< \rho_{\sig,n}\}$. Thus, defining $c^*\in \Ome_\sig$ as in Lemma \ref{lem_two_periods}, with $c=c^+$ or $c^-$ and $\tau=\pm\sig$, we see from Lemma \ref{lem_two_periods} that
\begin{equation} \label{spectral_inc}
\Spec(A_{c^*}) \supset \overline{I_{c}}\setminus E_\tau \supset \{\lam:|\lam| \leq \rho_{\sig,n}\}\setminus E_\tau.
\end{equation}
Applying Proposition \ref{quasiergodic}, it follows that, for all $n\in \N$,
$$
S_\sig \supset \{\lam:|\lam| \leq \rho_{\sig,n}\}\setminus H_\tau.
$$
The theorem follows since $\sup_{n} \rho_{\sig,n} = 1$ and $S_\sig$ is closed.
\end{proof}

\vspace{-2ex}

\noindent
%%%%%%%%%%%%%%%%%%%%%%%%%%%%%%%%%%%%%%%%%%%%%%%%%%%%%%%%%%%%%%%%%%%%%%%%%%%%
\ifpics{  %%%%% 30 png pictures (for PDFLatex only) %
\begin{center}
\begin{tabular}{cc}
\hspace*{-5ex} \includegraphics[width=0.55\textwidth]{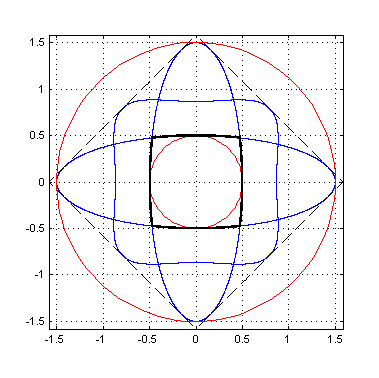}&
\hspace*{-5ex} \includegraphics[width=0.55\textwidth]{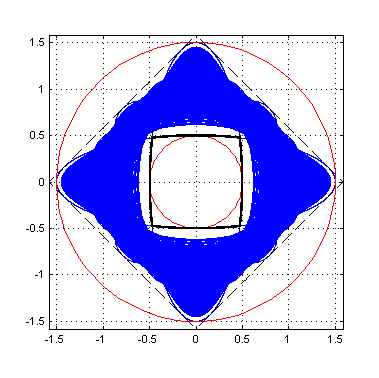}
\end{tabular}
\end{center}

\vspace*{-7ex}

}\fi %%%%% 30 png pictures (for PDFLatex only) %
%%%%%%%%%%%%%%%%%%%%%%%%%%%%%%%%%%%%%%%%%%%%%%%%%%%%%%%%%%%%%%%%%%%%%%%%%%%%

\begin{figure}[h]
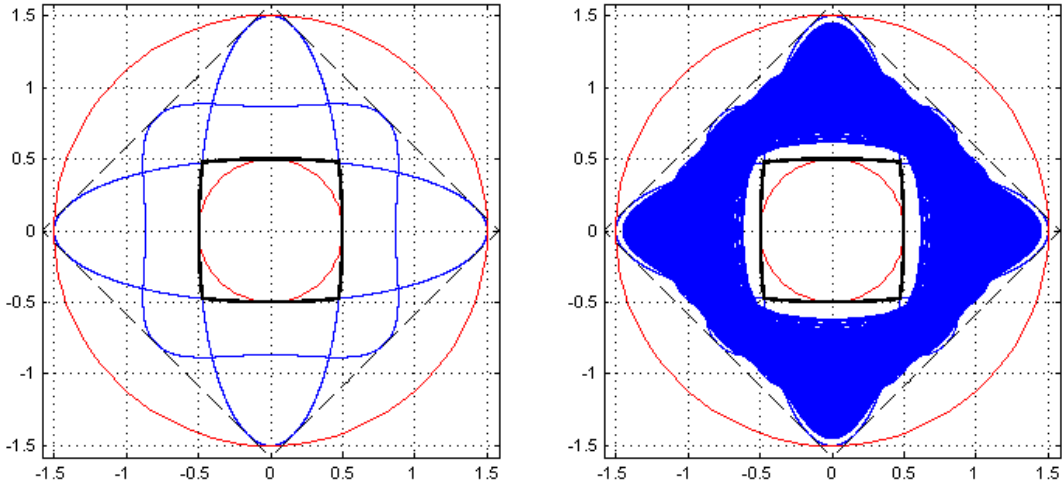

\caption{\small Plots of $\Spec(A_c)$ for the case when $c$ is periodic and $\sig=0.5$. The two plots show the sets $\pi_{N,\sig}\subset S_\sig$, the union of the spectra for all sequences $c$ of period $\leq N$, for $N=2$ (left) and $N=12$ (right). The two ellipses visible in the left-hand plot, the boundaries of $E_\sig$ and $E_{-\sig}$ defined in \eqref{ell1} and \eqref{ell2}, are the components of $\pi_{1,\sig}$; the other closed curve is $\Spec(A_c)$ for $c=\sig c^{(1,-)}$, i.e.\ $c_n=(-1)^n\sig$, given explicitly in Lemma \ref{lem_sf}. Also shown in each plot are the boundaries of the inclusion sets from Lemma \ref{lemmaA}, namely the circles of radius $1\pm\sig$, and, in dashed lines, the boundary of the set \eqref{numrange}. The boundary of $H_\sig$, which we conjecture is a hole in the spectrum $S_\sig$, is highlighted in a thicker line.}
\label{fig:30pics1}
\end{figure}

Lemma \ref{lemmaA} and Theorem \ref{maintheorem3} together establish that there is a hole in $S_\sig$ which is at least as big as $\{\lam:|\lam|<1-\sig\}$ and which is no larger than $H_\sig$. The numerical
computations we have been able to carry out are consistent with a hypothesis that the hole is precisely the set $H_\sig$, i.e.\ they are consistent with a hypothesis that $\Spec(A_c)\cap
H_\sig=\emptyset$ for every $c\in \Ome_\sig$, and hence for
every $c\in\cE_\sig$.

It should be pointed out, however, that these numerical computations are only for instances where $c$ is periodic, for which we have a characterisation of the spectrum in Lemma \ref{BIO}. Thus, strictly speaking, our calculations are evidence of the possibly weaker result that $\pi_{\infty,\sig} \cap H_\sig = \emptyset$; they become evidence that $\Spec(A_c)\cap
H_\sig=\emptyset$ with a hypothesis that $\pi_{\infty,\sig}$ is dense in the part of $S_\sig$ that is contained in the unit disc. This latter statement may or may not be true for $\sig\in (0,1)$, but we have shown in Theorem \ref{denseness} that it is true for $\sig=1$.

As an example of the numerical computations we have carried out, the right hand side of Figure~\ref{fig:30pics1} shows the union of $\Spec(A_c)$ over all periodic $c\in \Ome_\sig$ for which the period $N\leq 12$. It is clear from this figure that $\pi_{12,\sig}\cap H_\sig = \emptyset$ for $\sig=0.5$. We note that, rather than using the characterisation in Lemma \ref{BIO}, we use for these computations the standard Bloch-decomposition formula (e.g.\ \cite{LOTS}) that
\begin{equation} \label{spectrumperiodic}
\Spec(A_c) = \bigcup_{|\alpha|=1} \Spec(A^{(N,\mathrm{per})}_{c,\alpha}),
\end{equation}
where $A^{(N,\mathrm{per})}_{c,\alpha}$ is the $N\times N$ matrix defined in \eqref{finitematrix2} below.

\noindent
%%%%%%%%%%%%%%%%%%%%%%%%%%%%%%%%%%%%%%%%%%%%%%%%%%%%%%%%%%%%%%%%%%%%%%%%%%%%
\ifpics{  %%%%% 30 png pictures (for PDFLatex only) %
\begin{center}
\begin{tabular}{c}
\hspace*{-5ex} \includegraphics[width=0.6\textwidth]{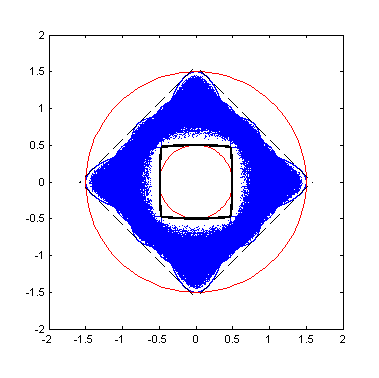}
\end{tabular}
\end{center}

\vspace*{-7ex}
}\fi %%%%% 30 png pictures (for PDFLatex only) %
%%%%%%%%%%%%%%%%%%%%%%%%%%%%%%%%%%%%%%%%%%%%%%%%%%%%%%%%%%%%%%%%%%%%%%%%%%%%

\begin{figure}[h]
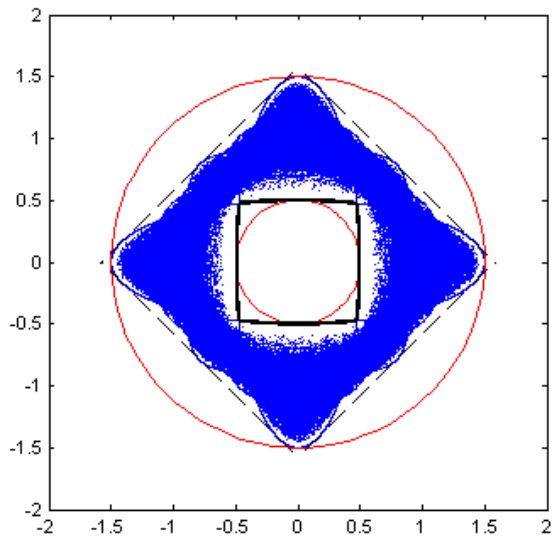

\caption{\small Plots of $10^5$ instances of $\Spec(A_{c,\alpha}^{(N,\mathrm{per})})$ for
$\sig=0.5$, computed as described in the text, with $N$ randomly chosen in the range $1\leq N\leq 100$. %Also shown in each plot are the boundaries of the inclusion sets from Lemma \ref{lemmaA}, namely the circles of radius $1\pm\sig$, and, in dashed lines, the boundary of the set \eqref{numrange}. The boundary of $H_\sig$, which we conjecture is a hole in the spectrum $S_\sig$, is highlighted in a thicker line.
}
\label{Hsigplot}
\end{figure}

It is not feasible to calculate $\pi_{N,\sig}$, the union of all $2^{N-1}$ periodic spectra of period $N$, for very much larger values of $N$. In Figure~\ref{Hsigplot} we sample $\pi_{100,\sig}$, for $\sig=0.5$, plotting the union of
the spectra of $10^5$ randomly chosen $N\times N$ matrices $A_{c,\alpha}^{(N,\mathrm{per})}$.
By randomly chosen we mean here that, in each realisation,  $N\in \{1,...,100\}$ is randomly chosen, with higher probabilities for the smaller matrix sizes, and then the vector $c=(c_1,...,c_N)$ is randomly chosen, with each $c_n=\pm \sig$ independent and identically distributed with $\mathrm{Pr}(c_n=\sig)=0.5$, and finally the phase factor $\alpha$ is randomly chosen from a uniform measure on the unit circle. We see in the figure that clearly, as they have to, the spectra are constrained to lie in the inclusion sets shown in Lemma \ref{lemmaA}. We also note that all the spectra lie outside $H_\sig$.

\section{Second proof of the main theorem}\label{secondproofsection}

In Theorem~\ref{maintheorem6} of this section we provide a second
proof that
\begin{equation}
\{\lam:|\lam|\leq 1\}\backslash H_\sig\subseteq
S_\sig\label{secondversion}
\end{equation}
for all $\sig\in (0,1)$. This proof has a lot in common with the
previous one, but it reveals more about the asymptotic behaviour of
the solutions of the second order recurrence relation for certain
choices of $c\in\Ome_\sig$. A key role in this section is played
by the sequence $c_e\in \Ome_\sig$ defined in Section \ref{sec:maps}, which sequence was central to the proof of Theorem \ref{maintheorem} that appears
in \cite{CCL} (see Theorem \ref{CCLprop3} above for more details).

We start with some calculations that do not depend on $\sig$.
Throughout this section $\widetilde{c}_n \in \{ \pm 1\}$ is defined
for all $n\geq 1$ by the rules $\widetilde{c}_1=1$,
$\widetilde{c}_{2n}=\widetilde{c}_{2n-1}\widetilde{c}_n$ and
$\widetilde{c}_{2n}+\widetilde{c}_{2n+1}=0$. (In other words, $\widetilde{c}_n=c_{e,n}$, for $n\geq 1$.)
 The first few values
are shown in Table~\ref{tabledata2}. We will obtain a bound on a
transfer matrix $T_{m,\lam}$  associated with this sequence and use
this bound to prove Theorem~\ref{maintheorem6}.

Let $u:\Z_+\to\C$ be the solution of $u_{n+1}=\lam u_n-\widetilde{c}_n
u_{n-1}$ such that $u_0=0$ and $u_1=1$, so that $u$ is the restriction to $\Z_+$ of the bi-infinite sequence $u_e$ already studied in \cite{CCL} and discussed in Section \ref{sec:maps}. We have observed already in Proposition \ref{CCLprop} that
$u_n$ is a polynomial of degree $n-1$ in $\lam$ with integer coefficients for
all $n\geq 2$. Similarly, if $v:\Z_+\to\C$ is the solution of
$v_{n+1}=\lam v_n-\widetilde{c}_n v_{n-1}$ such that $v_0=1$ and
$v_1=0$, it is easy to see that $v_n$ is a polynomial of degree $n-2$ with integer
coefficients for all $n\geq 2$.

\begin{table}[h]
\[
\begin{array}{rrrr}
n  &  \widetilde{c}_n       & u_n    & v_n   \\
\hline
1&\hspace{2em}1&1& \,\,\,\,0 \\
2&1&\lam& -1 \\
3&-1&\lam^2-1& -\lam  \\
4&-1&\lam^3& -\lam^2-1\\
5&1&\lam^4+\lam^2-1&-\lam^3 -2\lam\\
6&-1&\lam^5-\lam&  -\lam^4-\lam^2 +1\\
7&1&\lam^6+\lam^4-1& -\lam^5 -2\lam^3 -\lam \\
8&-1&\lam^7& -\lam^6 -\lam^4 -1\\
9&1&\hspace{2em}\lam^8+\lam^6+\lam^4-1&\hspace{2em}-\lam^7-2
\lam^5-2\lam^3-2\lam
\end{array}
\]
\caption{Values of $\widetilde{c}_n,\, u_n, \, v_n$ for $1\leq n\leq
9$.\label{tabledata2}}
\end{table}

One may check the computations of $u_m$ and $v_m$ in Table \ref{tabledata2} by confirming the
determinantal identity $u_mv_{m+1}-v_mu_{m+1}=\pm 1$ for all $m\geq
1$, the left hand side being a polynomial in $\lam$. Here we are referring to a determinantal identity for the transfer
matrix $T_{m,\lam}$, defined by
\[
T_{m,\lam}=\mtrx{ v_{m}&u_{m}\\ v_{m+1}&u_{m+1}},
\]
which transfers the data of any solution of $x_{n+1}=\lam
x_n-\widetilde{c}_n x_{n-1}$ from $\{ 0,1 \}$ to $\{ m,m +1 \}$ in
the sense that
\[
T_{m,\lam}\vctr{x_0\\ x_1}=\vctr{ x_m\\x_{m+1}}.
\]
It is easy to verify (see Lemma \ref{BIO} above) that
\begin{equation}
T_{m,\lam}=X_mX_{m-1}\ldots X_1\label{TSprod}
\end{equation}
where
\begin{equation}
X_r=\mtrx{0&1\\ -\widetilde{c}_r&\lam}\label{Sdetvalue}
\end{equation}
has determinant $\widetilde{c}_r\in \{\pm 1\}$ for every $r\geq 1$.

The proof of Theorem~\ref{maintheorem6} we will give below was motivated by numerical
evidence that
\[
\tr(T_{2^r,\lam}) =\lam^{2^r}-2
\]
holds for all $r\geq 1$ and all $\lam \in\C$; see Table \ref{tabledata3}. We prove this crucial
identity in Lemma~\ref{Ttau2}.

\begin{table}[h]
\[
\begin{array}{rr}
n  &  \hspace{2em}\tr(T_{n,\lam})= u_{n+1}+ v_n \\
\hline
1&\lam \\
2&\lam^2-2\\
3& \lam^3-\lam \\
4&\lam^4-2\\
5&\lam^5-\lam^3-3\lam\\
6&\lam^6-\lam^2\\
7& \lam^7-\lam^5-2\lam^3-\lam\\
8&\lam^8-2
\end{array}
\]
\caption{Values of $\tr(T_{n,\lam})$ for $1\leq n\leq
8$.\label{tabledata3}}
\end{table}

\begin{lemma}\label{Tdet1}
Let $T$ be a $2\times 2$ matrix with determinant $\del$ and trace
$\tau$. If $\tau^2\not= 4\del$ and $\gam$ is the absolute value of
the larger root of $z^2-\tau z+\del=0$ then there exists a constant
$b$ such that
\[
\norm T^r\norm\leq b\,\gam^r
\]
for all $r\geq 1$. If $\tau^2 = 4\del$ then for every $\eps >0$
there exists a constant $b_\eps$ such that
\begin{equation}
\norm T^r\norm\leq b_\eps(\gam+\eps)^r\label{jordan}
\end{equation}
for all $r\geq 1$.
\end{lemma}

\begin{proof}
The eigenvalues $z_\pm$ of $T$ are the roots $z$ of $z^2-\tau
z+\del=0$. The condition $\tau^2\not= 4\del$ implies that the
eigenvalues are distinct, so $T$ is diagonalizable -- there exists
an invertible matrix $B$ such that
\[
T=B\mtrx{ z_+&0\\0&z_-} B^{-1}.
\]
Therefore
\begin{eqnarray*}
\norm T^r\norm =\norm B\mtrx{ z_+^r&0\\0&z_-^r} B^{-1}\norm \leq  \norm B\norm \, \norm B^{-1}\norm \gam^r.
\end{eqnarray*}
The slightly worse bound (\ref{jordan}) is obtained when $\tau^2=
4\del$ because one has to use the Jordan canonical form for $T$.
\end{proof}

\begin{lemma}\label{Tdet3}
The identity $\det(T_{2^n,\lam}) =1$ holds for all $n\geq 1$ and all
$\lam\in\C$.
\end{lemma}

\begin{proof}
If $m\in\N$ then (\ref{TSprod}) and (\ref{Sdetvalue}) imply
\begin{eqnarray*}
\det(T_{2m,\lam})&=& \prod_{r=1}^m \det(X_{2r}X_{2r-1})\\
&=&\prod_{r=1}^m (\widetilde{c}_{2r}\widetilde{c}_{2r-1})\\
&=&\prod_{r=1}^m \widetilde{c}_{r}\\
&=& \det(T_{m,\lam}).
\end{eqnarray*}
It follows by induction that
\[
\det(T_{2^n,\lam})=\det(T_{1,\lam})=\widetilde{c}_1=1.
\]
\end{proof}

%We next summarize the results from \cite[Proposition~2.1]{CCL} that
%we will need.

%\begin{proposition}\label{CCLprop}
%Define $p_{i,j}$ for $i,j\in\N$ by the formula
%\[
%u_i=\sum_{j=1}^i p_{i,j}\lam^{j-1}
%\]
%with $p_{i,j}=0$ if $j>i$. Let $Y$ denote the set of $(i,j)\in \N^2$ such that $p_{i,j}\not= 0$. Then $p_{i,j}\in\{0,1,-1\}$ for all $i,\, j$ and $(i,j)\in Y$ if and only %if one of the following holds.\\
%(1) $i=j=1$;\\
%(2) $i$ and $j$ are both even and $(i/2,j/2)\in Y$;\\
%(3) $i$ and $j$ are both odd and $((i+1)/2,(j+1)/2)\in Y$;\\
%(4) $i$ and $j$ are both odd and $((i-1)/2,(j+1)/2)\in Y$.
%\end{proposition}

The following lemma depends on Proposition \ref{CCLprop} above, abstracted from \cite{CCL}, which notes properties of the integer coefficients $p_{i,j}$ of the polynomials
\[
u_i=\sum_{j=1}^i p_{i,j}\lam^{j-1}.
\]

\begin{lemma}\label{Ttau1}
The polynomial $u_m$ is even for odd $m$ and odd for even $m$. Its
leading term is $\lam^{m-1}$. If $m=2^n$ and $n\geq 2$ then
\begin{eqnarray}
u_m&=& \lam^{m-1},\label{ufortau1}\\
u_{m+1}&=&-1+\lam^{m/2}\sum_{r=0}^{m/4} \alp_r
\lam^{2r}\label{ufortau2}
\end{eqnarray}
where $\alp_r\in \{0,1,-1\}$ for all $r$.
\end{lemma}

\begin{proof}
The statements in the first two sentences may be proved by
induction, using the definition of $u_m$. We prove (\ref{ufortau1})
and (\ref{ufortau2}) for $m=2^n$ by induction in $n$, noting that
both hold for $n\leq 3$; see Table~\ref{tabledata2}. As in Proposition \ref{CCLprop}, let $Y$ denote the set of $(i,j)\in \N^2$ such that $p_{i,j}\not= 0$.

To prove (\ref{ufortau1}) suppose that $(2^{n+1},j)\in Y$.
Proposition~\ref{CCLprop} implies that $j$ is even and that
$(2^n,j/2)\in Y$. The inductive hypothesis now implies that
$j/2=2^n$, so $j=2^{n+1}$.

To prove (\ref{ufortau2}) suppose that $(2^{n+1}+1,j)\in Y$.
Proposition~\ref{CCLprop} implies that $j$ is odd and either
$(2^{n},(j+1)/2)\in Y$ or $(2^{n}+1,(j+1)/2)\in Y$. In the first
case we have already proved that $(j+1)/2=2^n$, so $j=2^{n+1}-1$. In
the second case the inductive hypothesis implies that $(j+1)/2  \geq
2^{n-1}+1$, so $j\geq 2^{n}+1$ or $(j+1)/2=1$, so $j=1$.

We finally have to evaluate the constant coefficient $\gam_m$ of
$u_m$ when $m=2^n+1$. This may be done by considering the defining
recurrence relation in the case $\lam=0$, namely
$\gam_{r+1}=-\widetilde{c}_r\gam_{r-1}$ subject to $\gam_0=0$ and
$\gam_1=1$. This implies that
\[
\gam_{2m+1}=\prod_{r=1}^m \widetilde{c}_{2r}
\]
for all $m\in\N$. Therefore $\gam_5=\gam_3=-1$ and
\begin{eqnarray*}
\gam_{8m+1}&=&\prod_{r=1}^m%
\left(\widetilde{c}_{8r}\widetilde{c}_{8r-2}%
\widetilde{c}_{8r-4}\widetilde{c}_{8r-6}\right)\\
&=&\prod_{r=1}^m%
\left(\widetilde{c}_{4r}\widetilde{c}_{8r-1}\widetilde{c}_{8r-2}%
\widetilde{c}_{4r-2}\widetilde{c}_{8r-5}\widetilde{c}_{8r-6}\right)\\
&=&\prod_{r=1}^m%
\left(\widetilde{c}_{4r}\widetilde{c}_{4r-2}\right)\\
&=& \gam_{4m+1}
\end{eqnarray*}
for all $m\in\N$. A simple induction now implies that
$\gam_{m}=-1$ for $m=2^n+1$ and all $n\geq 1$.
\end{proof}

\begin{lemma}\label{Ttau2}
If $m=2^n$ and $n\geq 2$ then
\begin{equation}
\tau=\tr(T_{m,\lam})=v_m+u_{m+1}=\lam^m-2\label{tauformula}
\end{equation}
for all $\lam\in\C$.
\end{lemma}

\begin{proof}
The proof uses the identity $v_mu_{m+1}-u_mv_{m+1}=1$ of
Lemma~\ref{Tdet3} together with the two identities proved in
Lemma~\ref{Ttau1}. These are identities within the commutative ring
$\Z(\lam)$ of all polynomials with integer coefficents in the
indeterminate quantity $\lam$, but they imply similar identities in
the commutative ring $\Z(\widehat{\lam})$ of all polynomials with
integer coefficients in an indeterminate quantity $\widehat{\lam}$
that satisfies the identity $\widehat{\lam}^{m-1}=0$. (Equivalently
one may start by disregarding all terms in the identities that
involve $\lam^r$ with $r\geq m-1$.) The identities then simplify to
\[
\widehat{u}_{m+1} = -1+p,\hspace{2em}
\widehat{v}_{m}\widehat{u}_{m+1}= 1,
\]
where
\[
p(\widehat{\lam})=\widehat{\lam}^{m/2}\sum_{r=0}^{m/4} \alp_r
\widehat{\lam}^{2r}
\]
satisfies $p^{2}=0$ in $\Z(\widehat{\lam})$. The second equation can
be solved for $\widehat{v}_m$, yielding
\[
\widehat{v}_{m}=-1-p
\]
and hence $\widehat{\tau}=-2$. Returning to the original variable
$\lam$ one deduces that
\[
\tau=-2+\sum_{r\geq m-1} \bet_r \lam^r.
\]
But (cf.\ Lemma \ref{Ttau1}) it is easily shown that $v_n$ is an even polynomial of degree $n-2$ for odd $n$. Thus, and by Lemma \ref{Ttau1}, $\tau =v_m+u_{m+1}$ is an even polynomial of degree $m$ with leading
coefficient $1$, so $\tau=\lam^m-2$.
\end{proof}

\begin{lemma}\label{Tdet2} Following the assumptions and notation of Lemma~\ref{Tdet1}, suppose that $\del=1$ and that there exist $m\in\Z_+$ and $\mu\in\C$ such that $\tau=\mu^m-2$. Then there exists a constant $b$ such that
\[
\norm T^r\norm\leq b\, 4^r \max(|\mu|^{rm},1)
\]
for all $r\geq 1$.
\end{lemma}

\begin{proof}\\
\textbf{Case~1.} If $|\mu|\leq 1$ it suffices to obtain bounds on
the solutions $z_\pm$ of $z^2-\tau z+1=0$ when $|\tau|\leq 3$. The
solutions satisfy
\[
|z_\pm|= \left| \frac{\tau}{2}\pm \sqrt{\frac{\tau^2}{4} -1}\right|
\leq\frac{3}{2}+\frac{\sqrt{13}}{2} <4.
\]
Lemma~\ref{Tdet1} now implies that $\norm T^r\norm\leq b \, 4^r$ for
all $r\geq 1$.

\textbf{Case~2.} If $|\mu|> 1$ it suffices to obtain bounds on the
solutions $z_\pm$ of $z^2-\tau z+1=0$ when $|\tau|\leq 3|\mu|^m$.
The solutions satisfy
\[
|z_\pm|= \left| \frac{\tau}{2}\pm \sqrt{\frac{\tau^2}{4} -1}\right|
<4|\mu|^m.
\]
Lemma~\ref{Tdet1} now implies that $\norm T^r\norm\leq b \,
4^r|\mu|^{rm}$ for all $r\geq 1$.
\end{proof}

\begin{lemma}\label{fillin}
Let $X_n$, $n\in\Z$, be a periodic sequence of $2\times 2$ matrices
with period $m$ and let $T_r=X_rX_{r-1}\ldots X_1$ for all $r\geq
1$. If there exist constants $b_0,\, \gam$ such that $\norm
(T_{m})^s\norm \leq b_0 \gam^{s}$ for all $s\geq 0$, then there
exists a constant $b_2$ such that $\norm T_r\norm\leq b_2
\gam^{r/m}$ for all $r\geq 1$.
\end{lemma}

\begin{proof}
Every $r\in\Z_+$ may be written in the form $r=sm+v$ where $s\geq 0$
and $0\leq v<m$. Using the identity $T_{sm}=(T_m)^s$, one obtains
\begin{eqnarray*}
\norm T_r\norm&=& \norm X_rX_{r-1}\ldots X_{sm+1}T_{sm}\norm\\
&=& \norm X_vX_{v-1}\ldots X_1(T_{m})^s\norm\\
&\leq & \norm X_vX_{v-1}\ldots X_1\norm b_0\gam^{sm}\\
&\leq & b_0b_1 \gam^{s}\\
&\leq & b_2 \gam^{r/m},
\end{eqnarray*}
where $b_2=b_0b_1$ and
\[
b_1=\max_{0\leq v\leq m-1}\{  \norm X_vX_{v-1}\ldots X_1\norm \}.
\]
\end{proof}

\begin{theorem}\label{maintheorem6}
One has
\begin{equation}
\{\lam:|\lam|\leq 1\}\backslash H_\sig\subseteq
S_\sig\label{strongform}
\end{equation}
for all $\sig\in (0,1)$.
\end{theorem}

\begin{proof} Given $\sig\in (0,1)$ we put $m=2^d$ where $d\in \N$ is
large enough to yield
\begin{equation}
\sig^{1/2} < h=4^{-1/m}.\label{criticalmbound}
\end{equation}
We use the identities
\[
\del=\det(T_{m,\mu})=1
\]
and
\[
\tau=\tr(T_{m,\mu})=\mu^{m}-2
\]
proved in Lemmas~\ref{Tdet3} and \ref{Ttau2} and valid for all
$\mu\in\C$. Let $c\in\Ome_\sig$ be the periodic sequence with period
$m$ such that $c_n=\sig \widetilde{c}_n$ for all $1\leq n\leq m$.
The main task is to prove that if $|\lam|<h$ then all
solutions $\xi:\Z\to\C$ of
\begin{equation}
\xi_{n+1}=\lam \xi_n-c_n\xi_{n-1}\label{xiiterate}
\end{equation}
decay exponentially as $n\to+\infty$. This will imply, by Lemma \ref{BIO},
and using the notations of that lemma, that
$$
I_c \supset \{\lam:|\lam| < h\}.
$$
Arguing as in the proof of Theorem \ref{maintheorem3}, it will then follow from Lemma \ref{lem_two_periods} and Proposition \ref{quasiergodic}
that
$$
S_\sig \supset \{\lam:|\lam| \leq h\} \setminus H_\sig ,
$$
this holding for any $h=4^{-1/m}$ such that \eqref{criticalmbound} holds and $m=2^d$, so that
$$
S_\sig \supset \{\lam:|\lam| < 1\} \setminus H_\sig .
$$
Since $S_\sig$ is closed, \eqref{strongform} will follow.

Thus it remains only to show that all solutions of \eqref{xiiterate} decay exponentially at $+\infty$.
To see that this holds,
define $x_n=\sig^{-n/2}\xi_n$ and $\mu=\sig^{-1/2}\lam$ so that (\ref{xiiterate}) may be rewritten in the form
\[
x_{n+1}=\mu x_n-\widetilde{c}_nx_{n-1}
\]
for $1\leq n\leq m$. Where $\theta=\max(1,|\mu|)$, Lemma~\ref{Tdet2} now yields
\[
\norm (T_{m,\mu})^r\norm \leq b\, 4^r\theta^{rm}
\]
for all $r\in \N$. Lemma~\ref{fillin} with $\gam=4\theta^m$ implies
\[
\norm T_{r,\mu}\norm \leq b\, 4^{r/m}\theta^r,
\]
and hence
\[
|x_r|\leq b_3\, 4^{r/m}\theta^r,
\]
again for all $r\in\N$. Hence, where $\phi = \max(\sig^{1/2},|\lam|)$,
\[
|\xi_r|\leq b_3\, 4^{r/m}\theta^r\sig^{r/2}=b_3 \left(
\phi h^{-1}\right)^r
\]
for all $r\in \N$. Since $0<\phi< h$, it follows that
$\xi$ decays exponentially.
\end{proof}

\section{Semi-infinite  and finite matrices}
All our results so far have focused on calculations of the spectrum of the bi-infinite matrix $A_c$. In this final section we say something about the spectrum of the semi-infinite matrix
\[
A^+_c= \left(\begin{array}{cccc}
0&1&&\\
c_1&0&1&\\
& c_{2}&0&\ddots\\
&&\ddots&\ddots
\end{array}\right)
\]
in the case that $c=(c_1,c_2,...)\in \{\pm \sigma\}^\N$ is pseudo-ergodic (contains every finite sequence of $\pm \sigma$'s as a consecutive sequence). We also say something (though have mainly unanswered questions) about the finite $N\times N$ matrices
\begin{eqnarray*}
A^{(N)}_c&=& \left(\begin{array}{ccccc}
0&1&&\\
c_1&0&1&\\
&c_{2}&0&\ddots\\
&&\ddots&\ddots& 1\\
&&& c_{N-1} & 0
\end{array}\right)
\end{eqnarray*}
and
\begin{eqnarray}
A^{(N,\mathrm{per})}_{c,\alpha}&=& \left(\begin{array}{ccccc}
0&1&&& \alpha c_N\\
c_1&0&1&\\
&c_{2}&0&\ddots\\
&&\ddots&\ddots& 1\\
\alpha^{-1} &&& c_{N-1} & 0
\end{array}\right).\label{finitematrix2}
\end{eqnarray}
Here $A^{(N)}_c$ is tridiagonal, $A^{(N,\mathrm{per})}_c$ is tridiagonal except for ``periodising'' entries in row $1$ column $N$ and row $N$ column $1$ (in these entries we assume that $|\alpha|=1$), and each $c_j = \pm \sig$: we have in mind particularly the random case where the $c_j$'s are independent and identically distributed random variables taking the values $\pm \sigma$.

Our main result on the spectrum of $A_c$, proved in the previous sections, is that it contains the set $\{\lambda:|\lambda| \leq 1\}\setminus H_\sigma$. We suspect that $H_\sigma$ is a genuine hole in the spectrum for $0<\sigma<1$, i.e.\ that $H_\sigma\cap S_\sigma = \emptyset$. We have not shown this result but have shown in Lemma \ref{lemmaA} the weaker result that $\{\lambda:|\lambda|< 1-\sigma\}\cap S_\sigma = \emptyset$. Our first result in this section is that this hole is not present in the spectrum of the semi-infinite matrix. The proof depends on recent results on semi-infinite pseudo-ergodic operators due to Lindner and Roch \cite{LR10}, derived using characterisations of the index of Fredholm operators, whose matrix representations are banded semi-infinite matrices, in terms of so-called ``plus indices'' of limit operators, these characterisations derived using $K$-theory results for $C^*$-algebras in \cite{RRJ}.

\begin{theorem} \label{thm:specsemi}
Suppose $c\in \{\pm \sigma\}^{\N}$ is pseudo-ergodic. If $\sigma=1$ then $\Spec(A_c^+) = S$. For all $\sigma\in (0,1]$, $\{\lambda:|\lambda|\leq 1\}\subset \Spec( A_c^+)$.
\end{theorem}
\begin{proof} In the case that $\sigma=1$ it is shown in \cite{CCL2} that $\Spec(A_c^+)=S$. Thus, for $\sigma=1$,
$$
\{\lambda:|\lambda|\leq 1\}\subset S= \Spec(A_c^+)
$$
follows from Theorem \ref{maintheorem} (or \cite[Theorem 2.3]{CCL}). %For $\sigma=0$, $A_c$ is just the shift operator and $A_c^+$  is a Toeplitz operator with spectrum $\Spec(A_c^+)=\{\lambda:|\lambda|\leq 1\}$.
For all $\sigma\in (0,1]$ it follows from \cite[Theorem 2.1]{LR10} that the essential spectrum of $A_c^+$, i.e.\ the set of $\lambda\in\C$ for which $A_c^+-\lambda I^+$ is not Fredholm (here $I^+$ is the identity operator on $\ell^2(\N)$), is the set $S_\sigma$. Thus and by Theorem \ref{maintheorem3},
$$
(\{\lambda:|\lambda| \leq 1\}\setminus H_\sigma)\subset S_\sigma \subset \Spec(A_c^+).
$$
It remains to show that $H_\sigma\subset \Spec(A_c^+)$. But, applying \cite[Theorem 2.4]{LR10} (note that the set $E_-(U,W)$ in the notation of \cite[Theorem 2.4]{LR10} is precisely the set $H_\sigma$ for this operator), it follows  that, for $\lambda \in H_\sigma$, either $A_c^+-\lambda I^+$ is not Fredholm or $A_c^+-\lambda I^+$ is Fredholm with index 1: in either of these cases $\lambda\in \Spec(A_c^+)$.
\end{proof}

Our other result in this section is to say something about the spectra (sets of eigenvalues) of the finite matrices $A^{(N)}_c$ and $A^{(N,\mathrm{per})}_{c,\alpha}$. The notations $\pi_{N,\sig}$ and $\pi_{\infty,\sig}$ are as defined in and above equation \eqref{eq:piinf} (and  $\pi_N$ and $\pi_\infty$ are our abbreviations for $\sigma=1$).

\begin{theorem} \label{thmfinite} If $0<\sig\leq 1$, $|\alpha|=1$ and $c\in \{\pm\sig\}^N$, then
$$
\Spec(A^{(N,\mathrm{per})}_{c,\alpha}) \subset \pi_{N,\sig}\subset \pi_{\infty,\sig}\subset S_\sig
$$
while
$$
\Spec(A^{(N)}_c)\subset \sqrt{\sig} \pi_{2N+2}\subset \sqrt{\sig} \pi_\infty\subset \sqrt{\sig}S.
$$
If $\lambda = x+\ri y$ is an eigenvalue of $A^{(N,\mathrm{per})}_{c,\alpha}$ then $1-\sig\leq |\lambda|\leq 1+\sig$ and $|x|+|y|\leq \sqrt{2(1+\sig^2)\,}$, while if $\lambda$ is an eigenvalue of $A^{(N)}_c$ then $|x|+|y|\leq 2\sqrt{\sig}$.
\end{theorem}
\begin{proof} The first of these statements is clear from the definition of $\pi_{N,\sig}$, \eqref{spectrumperiodic}, and Proposition \ref{quasiergodic} which gives that $\pi_{\infty,\sig}\subset S_\sig$. The second of these statements is shown for $\sig=1$ in \cite[Theorem 4.1]{CCL2}. The second statement follows for $0<\sig<1$ by the observation that, where $d\in \{\pm 1\}^N$, $c=\sig d\in \{\pm \sig\}^N$, and $D_N$ is the diagonal matrix with leading diagonal $(1, \sig^{1/2},\sig,...,\sig^{(N-1)/2})$, it holds that
$
D_N^{-1} A^{(N)}_c D_N = \sqrt{\sig}\,A^{(N)}_d.
$
The last sentence then follows from Lemma \ref{lemmaA}.
\end{proof}

\noindent Note that in the last sentence of the above theorem the condition $|x|+|y|\leq 2\sqrt{\sig}$ implies both that $|\lambda|\leq 1+\sig$ and that $|x|+|y|\leq \sqrt{2(1+\sig^2)\,}$.

\noindent
%%%%%%%%%%%%%%%%%%%%%%%%%%%%%%%%%%%%%%%%%%%%%%%%%%%%%%%%%%%%%%%%%%%%%%%%%%%%
\ifpics{  %%%%% 30 png pictures (for PDFLatex only) %
\begin{center}
\begin{tabular}{cc}
\includegraphics[width=0.5\textwidth]{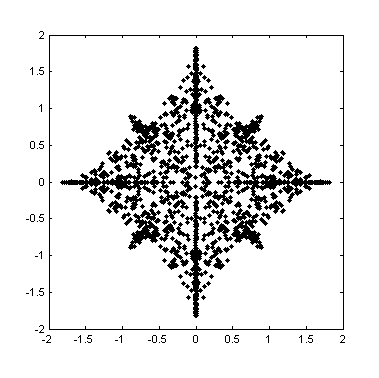}&
\includegraphics[width=0.5\textwidth]{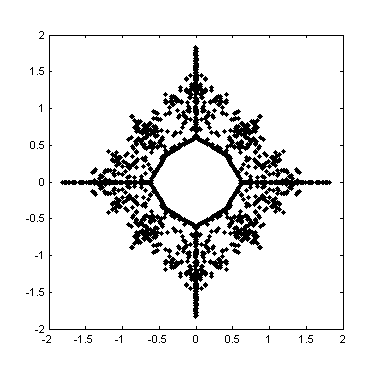}
\end{tabular}
\end{center}

\vspace*{-7ex}

}\fi %%%%% 30 png pictures (for PDFLatex only) %

%%%%%%%%%%%%%%%%%%%%%%%%%%%%%%%%%%%%%%%%%%%%%%%%%%%%%%%%%%%%%%%%%%%%%%%%%%%%
\begin{figure}[h]
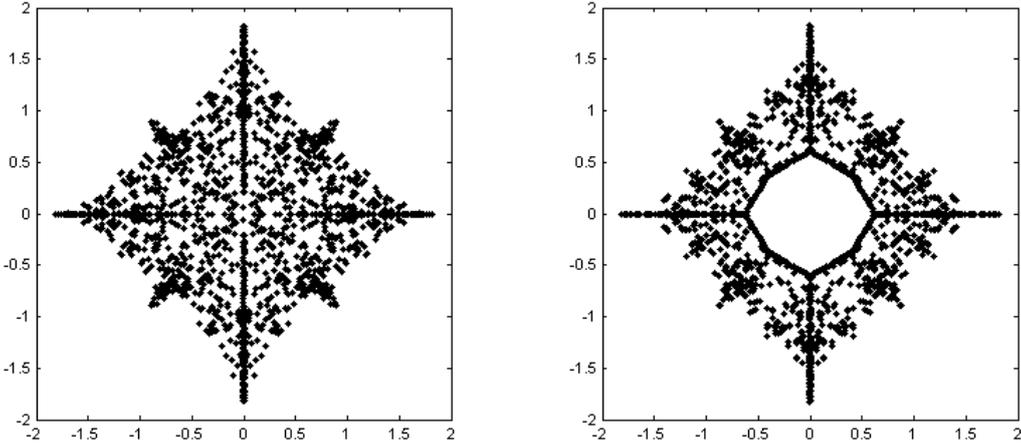

\caption{\small Plots of $\Spec(A_c^{(N)})$ (left) and $\Spec(A^{(N,\mathrm{per})}_{c,\alpha})$ (right) for a case when $N=2000$, $\sig = 0.9025$, $\alpha=1$, and the entries of the vector $c=(c_1,...,c_N)$ are independent and identically distributed with $\mathrm{Pr}(c_j = \pm\sigma)=0.5$ for each $j$ (the same vector $c$ is used in the two plots).}
 \label{fig:rand2000_09025}
\end{figure}

In Figure \ref{fig:rand2000_09025} we plot the spectra of $A_c^{(N)}$ and $A^{(N,\mathrm{per})}_{c,\alpha}$ for $N=2000$ and $\alpha=1$ for a typical realisation with the entries $c\in \{\pm \sig\}^N$ randomly chosen with the $c_j$ independently and identically distributed with $\mathrm{Pr}(c_j = \sig)=0.5$ and $\sig = 0.9025$ so that $\sqrt{\sig}=0.95$ (the several other realisations we have computed are very close in appearance to these plots). Theorem \ref{thmfinite} tells us that $\Spec(A_c^{(N)})\subset 0.95 \pi_\infty \subset 0.95 S$ and that $\Spec(A^{(N,\mathrm{per})}_{c,\alpha})\subset S_{0.9025}$, and that if $\lambda=x+\ri y$ is an eigenvalue of $A^{(N)}_c$ then $|x|+|y|\leq 1.9$, while if $\lambda$ is an eigenvalue of $A^{(N,\mathrm{per})}_{c,\alpha}$ then $0.075\leq |\lambda|\leq 1.9025$ and $|x|+|y|\leq \sqrt{2(1+\sig^2)\,}\approx1.905$.

It is clear from Figure \ref{fig:rand2000_09025} that Theorem \ref{thmfinite} is only the beginning of the story. We observe in the figure a hole in the spectrum of $A^{(N,\mathrm{per})}_{c,\alpha}$, but it is a hole of radius approximately 0.6, not 0.075,  with a large proportion of the eigenvalues positioned on the boundary of this hole, while outside the hole the spectra of $A^{(N,\mathrm{per})}_{c,\alpha}$ and $A^{(N)}_c$ appear near identical. The same qualitative behaviour is visible in Figure \ref{fig:rand2000_05}, which is a similar plot except that $\sig$ is reduced to 0.5 and we change the probability distribution, making it twice as likely that the entries of the vector $c$ are $-\sig$ rather than $\sig$. This change of probability distribution introduces an asymmetry, in particular an asymmetry in the hole in the spectrum (if we instead compute with $\mathrm{Pr}(c_j = \sigma)=1/2$ then typical realisations have spectra which are approximately invariant under the dihedral symmetry group $D_2$ of the square). Of course our methods, which are not probabilistic, have nothing to say about such asymmetries, indeed nothing, beyond Theorem \ref{thmfinite}, to say about the almost sure spectra of $A_c^{(N)}$ or $A^{(N,\mathrm{per})}_{c,\alpha}$ as $N\to\infty$.

\noindent
%%%%%%%%%%%%%%%%%%%%%%%%%%%%%%%%%%%%%%%%%%%%%%%%%%%%%%%%%%%%%%%%%%%%%%%%%%%%
\ifpics{  %%%%% 30 png pictures (for PDFLatex only) %
\begin{center}
\begin{tabular}{cc}
\includegraphics[width=0.5\textwidth]{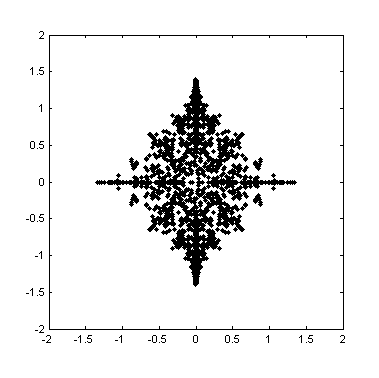}&
\includegraphics[width=0.5\textwidth]{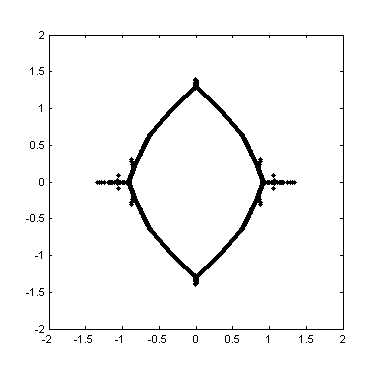}
\end{tabular}
\end{center}

\vspace*{-7ex}

}\fi %%%%% 30 png pictures (for PDFLatex only) %
%%%%%%%%%%%%%%%%%%%%%%%%%%%%%%%%%%%%%%%%%%%%%%%%%%%%%%%%%%%%%%%%%%%%%%%%%%%%
\begin{figure}[h]
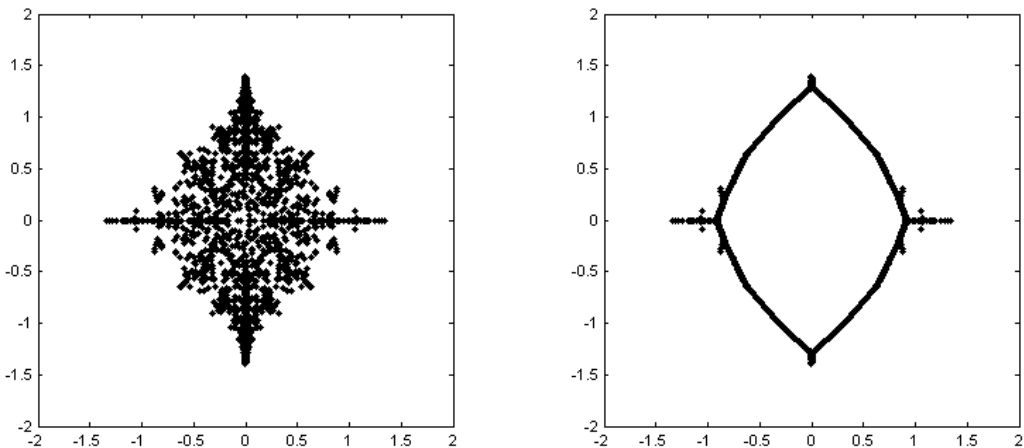

\caption{\small Plots of $\Spec(A_c^{(N)})$ (left) and $\Spec(A^{(N,\mathrm{per})}_{c,\alpha})$ (right) for a case when $N=2000$, $\sig = 0.5$ and the entries of the vector $c=(c_1,...,c_N)$ are independent and identically distributed with $\mathrm{Pr}(c_j = \sigma)=1/3$.}
\label{fig:rand2000_05}
\end{figure}

%\newpage

\parbox{3.5in}{S N Chandler-Wilde,\\
Dept. of Mathematics and Statistics,\\
University of Reading,\\
Berkshire, RG6 6BB,\\
UK}
\parbox{3.5in}{E B Davies,\\
Dept. of Mathematics,\\
King's College London,\\
Strand,\\
London,WC2R 2LS,\\
UK}

\end{document}